\documentclass[a4paper,10pt]{article}

\usepackage[ascii]{inputenc}
\usepackage{amsmath,amsfonts,amssymb,amsthm}
\usepackage{mathrsfs}
\usepackage{comment}
\usepackage{hyperref}
\usepackage{enumerate}

\usepackage{graphicx}

\usepackage{xcolor}

\newtheoremstyle{theorem}{5pt}{5pt}{\itshape}{}{\bfseries}{.}{.5em}{}

\theoremstyle{theorem}
\newtheorem{theorem}{Theorem}
\newtheorem{lemma}[theorem]{Lemma}
\newtheorem{corollary}[theorem]{Corollary}
\newtheorem{proposition}[theorem]{Proposition}
\newtheorem{definition}[theorem]{Definition}

\usepackage{titlesec}
\titleformat{\section}
{\Large\bfseries}{\thesection}{.7em}{\large}
\titlespacing*{\section}{0pt}{3.5ex plus 1ex minus .2ex}{2.3ex plus .2ex}
\titleformat{\subsection}
{\normalfont\bfseries}{\thesubsection}{.5em}{\normalsize}
\titlespacing*{\subsection}{0pt}{3.5ex plus 1ex minus .2ex}{2.3ex plus .2ex}

\DeclareMathOperator{\supp}{supp}
\newcommand{\abs}[1]{\left\lvert #1 \right\rvert}
\newcommand{\norm}[1]{\left\lVert #1 \right\rVert}
\newcommand{\R}{\mathbb{R}}
\newcommand{\C}{\mathbb{C}}

\newcommand{\p}{\partial}

\allowdisplaybreaks[4]

\begin{document}

\title{Non-Scattering Energies and\\ Transmission Eigenvalues in $\mathbb H^n$}
\author{Emilia Bl{\aa}sten\footnote{Department of Mathematics and Statistics, University of Helsinki, P.O. Box 64 (Pietari Kalmin Katu 5), 00014 Helsingin Yliopisto, FINLAND}{ } and
Esa V\!. Vesalainen\footnote{Mathematics and Statistics, {\AA}bo Akademi University, Domkyrkotorget 1, 20500 {\AA}bo, FINLAND}}
\date{}
\maketitle

\vspace*{-.5cm}
\begin{abstract}
We consider non-scattering energies and transmission eigenvalues of compactly supported potentials in the hyperbolic spaces $\mathbb H^n$. We prove that in $\mathbb H^2$ a corner bounded by two hyperbolic lines intersecting at an angle smaller than $180^\circ$ always scatters, and that one of the lines may be replaced by a horocycle. In higher dimensions, we obtain similar results for corners bounded by hyperbolic hyperplanes intersecting each other pairwise orthogonally, and that one of the hyperplanes may be replaced by a horosphere. The corner scattering results are contrasted by proving discreteness and existence results for the related transmission eigenvalue problems.

\medskip
\noindent{\bf Keywords}: hyperbolic geometry; interior transmission problem; corner; non-scattering

\medskip
\noindent{\bf Mathematics Subject Classification (2010)}: 35P25, 58J50, 35R30, 51M10, 58J05
\end{abstract}

\section{Introduction and main results}

We start with two subsections on scattering theory to fix notation and make relevant definitions as well as contrast the hyperbolic setting with the Euclidean. 
The second Subsection~\ref{HnScattering} contains our five main theorems starting from page \pageref{2Dthm}. We prove the corner scattering theorems in Sections \ref{CGO} and \ref{cornerScat}, and the interior transmission eigenvalue theorems in Section~\ref{ITE}.

\subsection{Background: the theory in $\mathbb R^n$} \label{RnScattering}

\paragraph{Scattering.} Let us start by briefly describing some simple short-range scattering theory in $\mathbb R^n$ at a fixed energy $\lambda\in\mathbb R_+$. Standard references on the topic are Chapter XIV in \cite{HormanderII} and the book \cite{Colton--Kress}. Thus, we are concerned with an incident wave $w$ solving
\[\left(-\Delta-\lambda\right)w=0\]
in $\mathbb R^n$, which is scattered by a short-range potential $V$, say a compactly supported $L^\infty$-function, thus creating a total wave $v$ solving the equation
\[\left(-\Delta+\lambda^\nu V-\lambda\right)v=0,\]
where $\nu\in\left\{0,1\right\}$,
again in $\mathbb R^n$. The value $\nu=0$ arises in quantum mechanics, and the resulting equation models e.g.\ electron and neutron beams, whereas the value $\nu=1$ arises from many classical models, in particular in acoustic and electromagnetic scattering. We call the case $\nu=0$ Schr\"odinger scattering and the case $\nu=1$ Helmholtz scattering. As it turns out, transmission eigenvalues behave differently for these two equations.

For definiteness, we are interested in incident waves which are Herglotz waves, i.e.\ waves of the form
\[w(x)=\int\limits_{S^{n-1}}g(\vartheta)\,e^{-i\sqrt\lambda\vartheta\cdot x}\,\mathrm d\vartheta,\]
where $g$ is an $L^2$-function on the unit sphere $S^{n-1}$. Equivalently, we may require $w$ to belong to the Agmon--H\"ormander space $B^\ast(\mathbb R^n)$, which consists of all functions $u\in L^2_{\mathrm{loc}}(\mathbb R^n)$ which satisfy
\[\sup_{R>1}\frac1R\int\limits_{B(0,R)}\left|u\right|^2<\infty,\]
where $B(0,R)\subset\mathbb R^n$ is the open ball of $\mathbb R^n$ of radius $R$ centered at $0$. Similarly, we will require $v\in B^\ast(\mathbb R^n)$. We recall here the classical fact that as $w$ solves an elliptic partial differential equation with real-analytic coefficients, $w$ must itself be real-analytic; see e.g.\ Section II.1.3 in \cite{Bers--John--Schechter}.

The waves $v$ and $w$ will of course be linked, the essential connection being that the scattered wave $u=v-w$ will have, in a sense, a special asymptotic behaviour at infinity as there exists an $L^2(S^{n-1})$-function $A$, depending on $w$ and $\lambda$, called the far-field pattern or the scattering amplitude, such that
\[u(x)=A\!\left(\frac x{\left|x\right|}\right)\frac{e^{i\sqrt\lambda x}}{\left|x\right|^{(n-1)/2}}+\mathrm{error},\]
as $\left|x\right|\longrightarrow\infty$. In terms of Agmon--H\"ormander spaces, this asymptotic expansion can be taken to mean that the difference of $u$ and the main term involving $A$ belongs to the space $\mathring B^\ast(\mathbb R^n)$, which consists of all functions $u_0\in B^\ast(\mathbb R^n)$ for which
\[\lim_{R\longrightarrow\infty}\frac1R\int\limits_{B(0,R)}\left|u_0\right|^2=0.\]

Now we have achieved a complete minimal picture of scattering theory. In typical applications, for instance, one creates incident waves $w$ (or rather, waves that would be $w$ if the wave motion took place in a uniform flat background), nature then solves the equation for $v$ with the perturbation $V$ present, and we measure the difference $u=v-w$ far away and thereby recover $A$, or in some cases the absolute value $\left|A\right|$. A typical applied inverse problem would then be to recover information about an unknown perturbation $V$ of the background from the measured far-field patterns $A$, or $\left|A\right|$. Indeed, this problem has been studied very intensively.

\paragraph{Non-scattering energies.} One natural question about the above setting is whether we can have $A\equiv0$ for some $w\not\equiv0$? If this happens, then we call $\lambda$ a non-scattering energy (or more precisely, a non-scattering energy for $V$). Essentially this means that for some special $w$ it happens that far away all the wave motion looks precisely as if the perturbation $V$ was not present.

The study of this rather special circumstance first arose from the study of linear sampling \cite{Colton--Kirsch} and factorization methods \cite{Kirsch1, Kirsch2}, all of which are methods for recovering qualitative information about $V$ from scattering measurements. Alas, these methods required that $\lambda$ is not a non-scattering energy. Thus, the question arose whether these energies are in some sense ``sparse'', say discrete or even empty?

\paragraph{Transmission eigenvalues.} The first stab at the problem was executed via transmission eigenvalues. In the above setting, if $V$ is indeed compactly supported, say for concreteness that $V$ vanishes outside some bounded non-empty open set $\Omega\subset\mathbb R^n$ which has a connected exterior, then Rellich's classical uniqueness theorem says that if indeed $A\equiv0$, then $v\equiv w$ far away, and by unique continuation, $v\equiv w$ outside $\Omega$. We thus obtain, by restricting $v$ and $w$ to $\Omega$, a non-trivial solution to the boundary value problem
\[\begin{cases}
\left(-\Delta+\lambda^\nu V-\lambda\right)v=0&\text{in $\Omega$},\\
\left(-\Delta-\lambda\right)w=0&\text{in $\Omega$},\\
v-w\in H^2_0(\Omega).
\end{cases}\]
Here $H^2_0(\Omega)$ is the closure of $C_{\mathrm c}^\infty(\Omega)$ in the usual Sobolev norm $\left\|\cdot\right\|_{H^2(\Omega)}$, and for smooth domains $\Omega$ the condition $v-w\in H^2_0(\Omega)$ reduces simply to having both $v\equiv w$ and $\partial_\nu v\equiv\partial_\nu w$ on $\partial\Omega$.

The values of $\lambda$ for which the above problem has a non-trivial $L^2$-solution are called (interior) transmission eigenvalues (for $\Omega$ and $V$). Thus, a non-scattering energy for $V$ (when $V$ is compactly supported) is always a transmission eigenvalue for $V$ and any suitable domain $\Omega$ outside of which $V$ vanishes. 

Transmission eigenvalues first appeared in the works \cite{Kirsch, Colton--Monk}. The most basic results about them is discreteness, obtained soon after \cite{Colton--Kirsch--Paivarinta}. Existence was obtained for fairly general potentials in \cite{Paivarinta--Sylvester}, and soon after the existence of infinitely many transmission eigenvalues in the Helmholtz case was shown \cite{Cakoni--Gintides--Haddar}. After that, transmission eigenvalues have been intensively studied, partly because they offer a new avenue for deriving information about a scatterer \cite{Cakoni--Colton--Monk, McLaughlin--Polyakov, McLaughlin--Polyakov--Sacks}. For more information and references on transmission eigenvalues, we recommend \cite{Cakoni--Haddar1, Cakoni--Haddar2, Colton--Paivarinta--Sylvester}.

\paragraph{Corner scattering.} For many years, it was unclear what exactly is the relationship between non-scattering energies and transmission eigenvalues. For reasonable radial potentials the notions were equivalent, for other potentials the situation was unclear. The issue was resolved in \cite{Blasten--Paivarinta--Sylvester}, where a large class of potentials with sharp rectangular corners were shown to not have any non-scattering energies, even though for many of those potentials one has an infinite discrete set of Helmholtz transmission eigenvalues. Further corner scattering results have been obtained in \cite{Elschner--Hu, Elschner--Hu2, Hu--Salo--Vesalainen, Paivarinta--Salo--Vesalainen, Xiao--Liu}. These techniques have also contributed back to the interior transmission problem recently, namely there is now insight into the distribution of energy of transmission eigenfunctions in a domain \cite{Blasten--Liu2017, Blasten--Li--Liu--Wang}. Also worth mentioning is the fact that these techniques also contribute to the inverse problem of determining information about a polyhedral medium scatterer from a single far-field measurement \cite{Blasten--Liu2017b, Elschner--Hu, Elschner--Hu2, Hu--Salo--Vesalainen}.

\subsection{Scattering in $\mathbb H^n$} \label{HnScattering}

\paragraph{Hyperbolic geometry.} Our main goal is to extend the above theory to the direction of the $n$-dimensional hyperbolic spaces $\mathbb H^n$ of constant curvature $-1$. For a general reference on hyperbolic scattering, we recommend \cite{Isozaki--Kurylev}. One concrete model for this space is the upper half-space model, in which we set $\mathbb H^n$ to be $\mathbb R^{n-1}\times\mathbb R_+$ with coordinates $\left\langle x',x_n\right\rangle=\left\langle x_1,x_2,\ldots,x_{n-1},x_n\right\rangle$. Near such a point the infinitesimal line element is divided by $x_n$. Among other things, this implies that the natural measure in $\mathbb H^n$ is $\mathrm d\mu=\mathrm dx_1\cdots\mathrm dx_{n-1}\,\mathrm dx_n/x_n^n$, where $\mathrm dx_1$, \dots, $\mathrm dx_{n-1}$, $\mathrm dx_n$ are Lebesgue measures in $\mathbb R$, \dots, $\mathbb R$, $\mathbb R_+$, respectively. The geometry of $\mathbb H^n$ gives in a natural fashion rise to a Laplace--Beltrami operator. By shifting this operator so that its spectrum coincides with $\left[0,\infty\right[$, we arrive at the free Schr\"odinger operator
\begin{equation}\label{H0def}
H_0=-x_n^2\left(\frac{\partial^2}{\partial x_1^2}+\ldots+\frac{\partial^2}{\partial x_{n-1}^2}+\frac{\partial^2}{\partial x_n^2}\right)+\left(n-2\right)x_n\frac\partial{\partial x_n}-\frac{(n-1)^2}4,
\end{equation}
which will replace the Euclidean $-\Delta$.

\paragraph{Scattering.} An incident wave at a given energy $\lambda\in\mathbb R_+$ would now be a solution $w$ to the equation
\[\left(H_0-\lambda\right)w=0\]
in $\mathbb H^n$. The natural hyperbolic analogue of the Agmon--H\"ormander space $B^\ast(\mathbb R^n)$ is given by the space $B^\ast(\mathbb H^n)$, which consists of all functions $u\in L^2_{\mathrm{loc}}(\mathbb H^n;\mathrm d\mu)$ such that
\[\sup_{R>e}\frac1{\log R}
\int\limits_{1/R}^R\int\limits_{\mathbb R^{n-1}}\bigl|u(x',x_n)\bigr|^2\mathrm dx'\,\frac{\mathrm dx_n}{x_n^n}<\infty.\]
Such solutions will in fact be exactly the solutions of the form
\[w=\mathscr F_0^\pm(\sqrt\lambda)^\ast\varphi\]
for some $\varphi\in L^2(\mathbb R^{n-1})$, where $\mathscr F_0^\pm(\sqrt\lambda)^\ast$ is the Poisson operator given by
\begin{multline*}
\bigl(\mathscr F_0^\pm(\sqrt\lambda)^\ast\varphi\bigr)(x',x_n)\\
=\frac{\sqrt{2\sqrt\lambda\sinh\pi\sqrt\lambda}}\pi
\mathscr F_0^\ast\!\left(\left(\frac{\left|\xi'\right|}2\right)^{\!\pm i\sqrt\lambda}x_n^{(n-1)/2}\,K_{i\sqrt\lambda}(\left|\xi'\right|x_n)\,\widehat\varphi(\xi')\right),
\end{multline*}
where $\widehat\varphi$ is the usual Euclidean Fourier transform and $\mathscr F_0^\ast$ is the usual Euclidean inverse Fourier transform in $L^2(\mathbb R^{n-1})$, which takes functions of $x'\in\mathbb R^{n-1}$ into functions of $\xi'\in\mathbb R^{n-1}$. Thus, $B^\ast(\mathbb H^n)$-solutions $w$ to the free Schr\"odinger equation are natural analogues of the Herglotz waves in $\mathbb R^n$. Again $w$ must be real-analytic as the solution of an elliptic equation with real-analytic coefficients \cite[Sect. II.1.3]{Bers--John--Schechter}.

As before, we model the perturbation of the flat uniform hyperbolic background by some potential function $V\in L^\infty_{\mathrm c}(\mathbb H^n)$, and the incident wave $w$ gives rise to a total wave $v\in B^\ast(\mathbb H^n)$ solving the perturbed equation
\[\left(H_0+\lambda^\nu V-\lambda\right)v=0\]
in $\mathbb H^n$, where $\nu=0$ or $\nu=1$, and we speak of Schr\"odinger or Helmholtz scattering, accordingly. A more general way of modeling perturbations is to have a different topology and metric in a subdomain in addition to a potential. Transmission eigenvalues were considered in this setting in \cite{Shoji, Morioka--Shoji} recently. We shall keep the geometry intact and only add a potential function.

Again, the solutions $v$ and $w$ will be linked by their asymptotic behaviour. This time the scattered wave $u=v-w$ has an asymptotic expansion of the form
\[u(x)=A(x')\,x_n^{(n-1)/2-i\sqrt\lambda}+\text{error}\]
as $x_n\longrightarrow0+$ for some $A\in L^2(\mathbb R^{n-1})$ which again deserves to be called the far-field pattern or the scattering amplitude.
More precisely, this holds in the sense that
\[\lim_{R\longrightarrow\infty}\frac1{\log R}\int\limits_{1/R}^1\int\limits_{\mathbb R^{n-1}}\left|u(x',x_n)-A(x')\,x_n^{(n-1)/2-i\sqrt\lambda}\right|^2\mathrm dx'\,\frac{\mathrm dx_n}{x_n^n}=0,\]
and
\[\lim_{R\longrightarrow\infty}\frac1{\log R}\int\limits_{1}^R\int\limits_{\mathbb R^{n-1}}\left|u(x',x_n)\right|^2\mathrm dx'\,\frac{\mathrm dx_n}{x_n^n}=0.\]

\paragraph{Non-scattering energies.} We may ask whether it is possible that $A\equiv0$ for some $w\not\equiv0$? If affirmative, we again call $\lambda$ a non-scattering energy for $V$. It turns out that this is equivalent with the scattered wave $u=v-w$ belonging to the space $\mathring B^\ast(\mathbb H^n)$ which consists of those functions $u\in B^\ast(\mathbb H^n)$ for which
\[\lim_{R\longrightarrow\infty}\frac1{\log R}\int\limits_{1/R}^R\int\limits_{\mathbb R^{n-1}}\left|u(x',x_n)\right|^2\mathrm dx'\,\frac{\mathrm dx_n}{x_n^n}=0.\]

\begin{definition}
\label{nonScatteringEnergyHn}
A number $\lambda\in\mathbb R_+$ is called a non-scattering energy for the potential $V\in L^\infty_{\mathrm c}(\mathbb H^n)$ if there exist functions $v\not\equiv0$ and $w\not\equiv0$ belonging to $B^\ast(\mathbb H^n)$ and solving
\[\left(H_0+\lambda^\nu V-\lambda\right)v=0\quad\text{and}\quad\left(H_0-\lambda\right)w=0\]
in $\mathbb H^n$, and for which $v-w\in\mathring B^\ast(\mathbb H^n)$.
\end{definition}

Our main theorems about these objects will include analogues of previous Euclidean corner scattering results, but we can also prove similar results for more exotic hyperbolic corners.
\begin{definition} \label{coneDef}
  An open set $\mathcal C \subset \mathbb H^n$ is called a cone if
  there is a vertex $x_0\in\partial\mathcal C$ such that for any
  $x\in\mathcal C$, the open ray from $x_0$ to $x$ belongs to
  $\mathcal C$. The cone $\mathcal C$ is admissible if
  \begin{enumerate}[{case} i)]
    \item $n=2$ and it is delimited by two non-parallel rays starting from
      $x_0$, and $\mathcal C$ is on the side where the opening angle
      is less than $\pi$, or\label{2dCorner}
    \item $n\geqslant2$ and it is delimited by $n$ hyperbolic
      hyperplanes all of which intersect pairwise at an angle of
      $\pi/2$ at their common vertex.\label{ndCorner}
  \end{enumerate}
\end{definition}

\begin{theorem}\label{2Dthm}
Let $\Omega \subset \mathbb H^2$ be a non-empty bounded open set such that the interior of $\mathbb H^2 \setminus \Omega$ is connected. Assume that there is a nonempty open set $B$ such that $B \cap \Omega = B \cap \mathcal C$ for some open hyperbolic cone $\mathcal C$ as in Definition~\ref{coneDef}.

Let $V \in L^\infty(\mathbb H^2)$ vanish outside of $\Omega$ and assume that there is $\varphi \in C^\alpha(\mathbb H^2)$, $\alpha\in\mathbb R_+$, of compact support such that $V = \chi_{\mathcal C} \varphi$ in $B$. If $\varphi(x_0) \neq 0$, then the potential $V$ has no non-scattering energies.
\end{theorem}

\begin{theorem}\label{3Dthm}
Let $n\geqslant 2$ be an integer and let $\Omega \subset \mathbb H^n$ be a non-empty bounded open set such that the interior of $\mathbb H^n \setminus \Omega$ is connected. Assume that there is a nonempty open set $B$ such that $B \cap \Omega = B \cap \mathcal C$ for some open hyperbolic cone $\mathcal C$ as in case~\ref{ndCorner}) of Definition~\ref{coneDef}.

Let $V \in L^\infty(\mathbb H^n)$ vanish outside of $\Omega$ and assume that there is $\varphi \colon \mathbb H^n \longrightarrow \C$ of compact support such that $V = \chi_{\mathcal C} \varphi$ in $B$. Moreover assume either $n=2$ with $\varphi\in C^\alpha(\mathbb H^2)$ and $\alpha\in\mathbb R_+$, or $n=3$ with $\varphi\in C^\alpha(\mathbb H^3)$ and $\alpha\in\left]1/4,\infty\right[$, or else $n\geqslant 3$ and $\varphi \in H^{s,r}(\mathbb H^n)$ with $r\in\left[ 1,\infty\right[$ and $s\in\left]n/r,\infty\right[$. If $\varphi(x_0) \neq 0$, then the potential  $V$ has no non-scattering energies.
\end{theorem}

In addition, the hyperbolic space $\mathbb H^n$ has natural objects called horospheres, also called horocycles in two dimensions. In the half-space model of $\mathbb H^n$, a horosphere means either a hyperplane parallel to the hyperplane $\mathbb R^{n-1}\times\left\{0\right\}$ at infinity, or a sphere tangent to $\mathbb R^{n-1}\times\left\{0\right\}$. In the ball model of $\mathbb H^n$, a horosphere means a sphere tangent to the sphere at infinity.

We can generalize theorem~\ref{3Dthm} to the case where one of the sides of the hyperbolic cone is a horocycle or horosphere. Simply choose $\Phi$ as the upper half-space coordinates of $\mathbb H^n$ shifted to $\R^{n-1}\times{]{-1,\infty}[}$.

\begin{theorem}\label{confThm}
Theorems \ref{2Dthm} and \ref{3Dthm} stay true if $\mathcal C$ is defined as follows instead: Assume that there is a conformal map $\Phi \colon B \longrightarrow \R^n$ such that $\Phi(x_0) = 0$ for some $x_0\in B$, and that $\mathcal C$ is the preimage under $\Phi$ of a Euclidean cone with opening angle less than $\pi$ and vertex $0$ in two dimensions, or $\mathcal C = \Phi^{-1}({]{0,\infty}[}^n)$ in higher dimensions.
\end{theorem}

Strictly speaking, the proofs of these corner scattering results are
mostly local, similarly to those in \cite{Hu--Salo--Vesalainen}, in
the sense that the only input from scattering theory consists of the
scattering solutions to the free and perturbed equations and the
suitable Rellich type theorem which is used to show that the solutions
coincide outside the corner near the vertex of the corner. After this
the equations themselves are immediately restricted to a small
neighbourhood of the vertex, and scattering theory is not mentioned
again in the proofs. Thus, the above corner scattering results hold in
any reasonable scattering setting as long as these first steps can be
taken and lead to the same equations in a small neighbourhood of the
vertex. See Proposition~\ref{finalProp} for a precise statement. The 
theorems are proven using these observations and by straightening 
the corner with suitable coordinates. See Figure~\ref{fig1}.
\begin{figure}
\begin{center}
\includegraphics{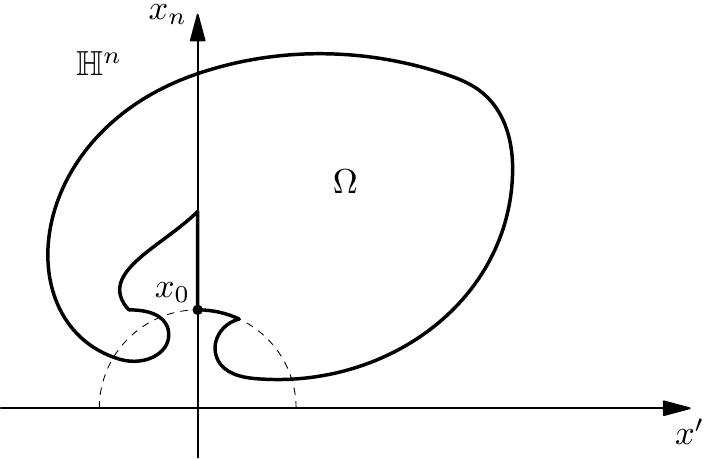}
\includegraphics{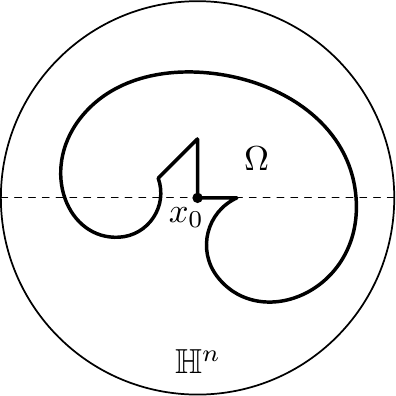}
\caption{The Poincar\'e upper half-space and ball models of $\mathbb H^n$ and a scatterer~$\Omega$.}
\label{fig1}
\end{center}
\end{figure}

\paragraph{Transmission eigenvalues.}
In the hyperbolic setting we also have a Rellich type theorem (see e.g.\ Theorem 2.10 in \cite{Isozaki--Kurylev}). Again, if $\Omega$ is a non-empty bounded open set in $\mathbb H^n$ with a connected exterior, then this Rellich type theorem and unique continuation imply that the restrictions of $v$ and $w$ into $\Omega$ solve the system
\[\begin{cases}
(H_0+\lambda^\nu V-\lambda)v=0,\\
(H_0-\lambda)w=0,\\
v-w\in H^2_0(\Omega),
\end{cases}\]
where the space $H^2_0(\Omega)$ is the same as $H^2_0(\Omega)$ in the Euclidean sense, if we employ the non-canonical embedding $\Omega\subset\mathbb H^n\subset\mathbb R^n$. Of course, the natural Sobolev norm in $\mathbb H^n$ would involve extra powers of $x_n$, but in a bounded domain the Euclidean and hyperbolic norms will be equivalent.

We may define transmission eigenvalues as in the Euclidean case.
\begin{definition}
Let $\Omega\subset\mathbb H^n$ be a bounded non-empty open set, and let $V\in L^\infty(\Omega)$. A number $\lambda\in\mathbb C$ (or rather $\lambda\in\mathbb C^\times$ when $\nu=1$) is called a transmission eigenvalue for the potential $V$ (and $\Omega$) if there exist $L^2(\Omega)$-functions $v\not\equiv0$ and $w\not\equiv0$ solving
\[\left(H_0+\lambda^\nu V-\lambda\right)v=0\quad\text{and}\quad\left(H_0-\lambda\right)w=0\]
in $\Omega$, and satisfying $v-w\in H^2_0(\Omega)$. We speak of Schr\"odinger transmission eigenvalues and Helmholtz transmission eigenvalues in the cases $\nu=0$ and $\nu=1$, respectively.
\end{definition}

Our two main theorems about transmission eigenvalues are better given separately. First, the Schr\"odinger case:
\begin{theorem}\label{schrodinger-transmission-eigenvalues}
Let $\Omega\subset\mathbb H^n$ be a bounded non-empty open set, and let $V\in L^\infty(\Omega)$ be bounded away from zero and take only positive real values. Then then set of $\lambda\in\mathbb C$ for which the system
\[\begin{cases}
(H_0+V-\lambda)v=0&\text{in $\Omega$,}\\
(H_0-\lambda)w=0&\text{in $\Omega$,}\\
v-w\in H^2_0(\Omega),
\end{cases}\]
has a solution $v,w\in L^2(\Omega)$ with $v\not\equiv0$ and $w\not\equiv0$, is a discrete subset of $\mathbb R_+$ which can only accumulate to $+\infty$. For each fixed $\lambda$, the space of solutions is finite-dimensional. Furthermore, if $N\in\mathbb Z_+$ and in some open ball $B\subseteq\Omega$ the potential $V$ is sufficiently large (depending on $n$, $N$ and the hyperbolic radius of the ball), then there are at least $N$ such transmission eigenvalues $\lambda$, counting multiplicities.
\end{theorem}
\noindent
We emphasize that for real-valued potentials $V$, all Schr\"odinger transmission eigenvalues are real.

In the Helmholtz case, we get a result of a somewhat different shape:
\begin{theorem}\label{helmholtz-transmission-eigenvalues}
Let $\Omega\subset\mathbb H^n$ be a bounded non-empty open set, and let $V\in L^\infty(\Omega)$ be bounded away from both zero and one and take values either only from $\left]-\infty,0\right[$, or only from $\left]0,1\right[$. Then then set of $\lambda\in\mathbb R^\times$ for which the system
\[\begin{cases}
(H_0+\lambda V-\lambda)v=0&\text{in $\Omega$,}\\
(H_0-\lambda)w=0&\text{in $\Omega$,}\\
v-w\in H^2_0(\Omega),
\end{cases}\]
has a solution $v,w\in L^2(\Omega)$ with $v\not\equiv0$ and $w\not\equiv0$, is an infinite discrete subset of $\mathbb R$. For each fixed $\lambda\in\mathbb R^\times$, the set of solutions is finite-dimensional. Furthermore, if $V$ takes only negative values, then the set of $\lambda$ does not contain negative elements.
\end{theorem}


\section{Complex geometrical optics solutions} \label{CGO}

Complex geometrical optics solutions \cite{Sylvester--Uhlmann} and their error estimates are a fundamental tool for studying corner scattering. We will choose suitable coordinates in $\mathbb H^n$ and conjugate the free operator $H_0$ with a suitable function to show the existence of these solutions. This allows us to bring past techniques of \cite{Blasten--Paivarinta--Sylvester, Paivarinta--Salo--Vesalainen, Hu--Salo--Vesalainen} into the hyperbolic setting.

By \emph{conjugating} the free operator $H_0$ from \eqref{H0def} on page \pageref{H0def} with a suitable function $K$ we get
\begin{equation} \label{conjH0}
  K^{-(n+2)/2}\, H_0 \,\bigl(K^{(n-2)/2}f\bigr) = (-\Delta + Q_K)\,f,
\end{equation}
where $Q_K$ is a new potential function depending on $K$ and $\Delta$ is the Euclidean Laplacian on the half-space model $\R^{n-1}\times\R_+$ of $\mathbb H^n$. A similar formula holds when instead of $H_0$ we have a \emph{Laplace--Beltrami operator} $H_K$ for any metric conformal to the Euclidean one.

A Riemannian metric $g$ on an open nonempty set $U\subset\R^n$ can be written as $g\colon U\times\R^n\times\R^n\longrightarrow\R$, $(x,a,b)\mapsto g_x(a,b)$ where $g_x$ is an inner product. It is conformal to the Euclidean metric if $g_x(a,b) = \lambda^2(x) \,a\cdot b$ for some smooth function $\lambda\colon U \longrightarrow \R_+$. To simplify formulas we write $K = 1/\lambda$. The metric can be represented by a matrix with components $(g_x)_{ij} = g_x(e_i,e_j)$ where $e_1$, $e_2$, \dots, $e_n$ are the standard unit vectors of $\R^n$. Indeed, under such coordinates we have $(g_x)_{ii} = K^{-2}(x)$ and $(g_x)_{ij}=0$ when $i\neq j$. Hence its determinant is $\lvert g_x\rvert = K^{-2n}(x)$. We can then define the $K$-divergence at $x\in U$ by
\[
  \nabla_K \cdot X(x) = \frac{1}{\sqrt{\abs{g_x}}} \sum_{j=1}^n \partial_j( \sqrt{\abs{g_x}} \,X_j(x) ) = K^n(x) \,\nabla \cdot (K^{-n}(x)\, X(x))
\]
where $x \longmapsto X(x) \in \R^n$ is once differentiable on $U$. Here $\nabla\cdot$ is the Euclidean divergence. Denote by $g_x^{ij}$ the component $(g_x^{-1})_{ij}$ of the inverse matrix of $g_x$. The gradient is
\[
  \nabla_K f(x) = \sum_{i=1}^n\sum_{j=1}^n e_i\, g_x^{ij} \,\partial_{x_j} f(x) = K^2(x)\, \nabla f(x)
\]
for any differentiable function $f$. Hence the hyperbolic Laplace--Beltrami operator is
\[
  -\Delta_K f = - \nabla_K \cdot (\nabla_K f) = -K^2 \,\Delta f + (n-2)\, K \,\nabla K \cdot \nabla f
\]
in these coordinates. We remark that $K(x) = x_n$ in the upper half-space coordinates of $\mathbb H^n$, and $K(x) = 2/(1-\abs{x}^2)$ in the Poicar\'e disk. We indeed have $H_0 = -\Delta_K - (n-1)^2/4$ with $K(x)=x_n$ in \eqref{H0def}.

The following lemma generalizes \eqref{conjH0} to other conformal coordinates. It suggests the formula
\[
  u_0(x) = K(x)^{(n-2)/2}\,e^{x\cdot\rho}\,\bigl(1 + \psi(x)\bigr)
\]
for the complex geometrical optics solutions.

\begin{lemma} \label{conjugation-lemma}
Let $U\subseteq\mathbb R^n$ be a non-empty open set, let $K\in C^2(U)$ take only positive real values, and let $H_K$ be the partial differential operator given there by
\begin{equation}\label{HKformula}
H_K = - K^2 \,\Delta + (n-2) \,K \,\nabla K \cdot \nabla - \frac{(n-1)^2}{4}.
\end{equation}
Then
\begin{align}
&K^{-(n+2)/{2}} \,H_K \,\bigl( K^{(n-2)/2} f \bigr) \notag\\
&\qquad= -\Delta f+ \frac{(n-2)\,(n\abs{\nabla K}^2 - 2\,K\,\Delta K) - (n-1)^2}{4\,K^2} \,f,\label{DeltaMConjEQ}
\end{align}
for any $f\in C^2(U)$.
\end{lemma}

\begin{proof}
For a given real number $s$ we have $\nabla K^s = s \,K^{s-1} \,\nabla K$ and
\[
\Delta K^s = \nabla \cdot (s \,K^{s-1} \,\nabla K ) = s\,(s-1)\, K^{s-2} \abs{\nabla K}^2 + s \,K^{s-1}\, \Delta K .
\]
The rest follows from the Leibniz rule.
\end{proof}

We are going to apply techniques from \cite{Blasten--Paivarinta--Sylvester, Paivarinta--Salo--Vesalainen, Hu--Salo--Vesalainen} to construct complex geometrical optics solutions $u_0$. We start with the equation $(H_0+\lambda^\nu \,V-\lambda)u_0=0$ in $\mathbb H^n$. After a suitable choice of coordinates in a non-empty open set $U \subseteq \R^n$ this equation will become
\[
  \left(-K^2 \,\Delta + (n-2)\,K\, \nabla K \cdot \nabla - \frac{(n-1)^2}{4} + \lambda^\nu\, V - \lambda \right)u_0 = 0
\]
and $K\colon U\longrightarrow\R_+$ is chosen as described before Lemma~\ref{conjugation-lemma}. Setting $u_0 = K^{(n-2)/2}u$ gives us then
\[
  \left( -\Delta + \frac{(n-2)\,(n\abs{\nabla K}^2 - 2\,K\,\Delta K) - (n-1)^2 +4\left(\lambda^\nu\, V - \lambda\right)}{4\,K^2}\right) u = 0
\]
by the lemma.

For the purposes of corner scattering it is enough to show the existence of $u$ in a small neighbourhood $B\subseteq U$ of the corner. We shall assume that the potential $V$ is equal to $\chi_{\mathcal C}\, \varphi$ in $U$. Here $\mathcal C \subset \R^n$ is a polyhedral Euclidean cone and $\varphi \colon \R^n \longrightarrow \C$ a H\"older continuous function. The function $\Phi \in C^\infty_{\mathrm c}(U)$ below is a cut-off function which restricts the problem to the neighbourhood of interest.

\begin{lemma}\label{potEstims}
Let $n\in\mathbb Z_+$ with $n\geqslant2$ and let $\mathcal C \subset \R^n$ be an open polyhedral cone.
Consider the potential $V = \chi_{\mathcal C} \,\varphi$ where $\varphi \colon \R^n \longrightarrow \C$ is compactly supported and in a function space $X$. 
Here, in
\begin{enumerate}[{case} a)]
	\item $X = C^\alpha(\R^n)$ for some $\alpha\in\mathbb R_+$, or in\label{caseHold}
	\item $X = H^{s,r}(\R^n)$ for some $r \in\left[1,\infty\right[$ and $s\in\left]n/r,\infty\right[$.\label{caseSob}
\end{enumerate}

Assume that $K \colon U \longrightarrow \R_+$ is a smooth real-valued function with a positive lower bound and defined in a non-empty open set $U \subseteq \R^n$.
Let $\Phi \in C^\infty_{\mathrm c}(U)$ and define $Q\colon\R^n\longrightarrow\C$ by extending
\begin{equation}\label{qDef}
  \Phi\, \frac{(n-2)(n\abs{\nabla K}^2 - 2K\,\Delta K) - (n-1)^2 +4(\lambda^\nu \,V - \lambda)}{4K^2}
\end{equation}
by zero outside of $\supp \Phi$, and where $\nu\in\{0,1\}$ and $\lambda\in\mathbb R_+$ are constants.

Then, in the case~\ref{caseHold}), if $q\in\left[1,2\right]$ and $s \in\left[0, \min(1/q, \alpha)\right[$, we have
\begin{equation}\label{QSobMult}
  Q \in H^{s,q}(\R^n), \qquad \norm{Q f}_{H^{s,q}(\R^n)} \leqslant C \norm{f}_{H^{s,q'}(\R^n)}
\end{equation}
for $1/q + 1/q'=1$. In the case~\ref{caseSob})
\[
  \widehat Q \in B^1_{q,1}, \qquad \norm{\mathscr{F}\{Qg\}}_{B^1_{q,1}} \leqslant C \norm{\widehat g}_{B^{-1}_{q,\infty}}
\]
for any $q\in\left]1,\infty\right[$.
\end{lemma}

\begin{proof}
The polyhedral cone $\mathcal C$ can be expressed as a finite intersection of half-spaces. Hence
\[
  \chi_{\mathcal C} = \prod_{j=1}^m \chi_{H_j}
\]
for a set of half-spaces $H_j$. We can write $Q = \Psi_1 + \prod_j \chi_{H_j} \Psi_2 \varphi$ for some $\Psi_1, \Psi_2 \in C^\infty_{\mathrm c}(U)$.

\smallskip
Consider the case~\ref{caseSob}) first. The proof is more or less the same as in Lemma 3.4 of \cite{Hu--Salo--Vesalainen} and based on estimates in \cite{Blasten--Paivarinta--Sylvester}. In particular note that if $\supp a \subset B(0, R)$ then
\begin{align*}
  \norm{\widehat a}_{B^1_{q,1}} &\leqslant C R \norm{\widehat a}_{L^q},\\
	\norm{\widehat{ag}}_{B^1_{q,1}} &\leqslant 2 R^2 \norm{\widehat a}_{L^1} \norm{\widehat g}_{B^{-1}_{q,\infty}},\\
	\norm{\mathscr{F}\{\chi_{H_j} a\}}_{B^1_{q,1}} &\leqslant C \norm{\widehat a}_{B^1_{q,1}}.
\end{align*}
The first one of these implies that the Fourier transforms of $C^\infty_{\mathrm c}(U)$-functions are in the Besov space $B^1_{q,1}$. Thus the first and last imply that $\mathscr{F} Q \in B^1_{q,1}$ if $\mathscr{F}\{\Psi_2 \varphi\} \in L^q$. Similarly, the second and last one imply the required mapping properties for $Q$ if $\mathscr{F}\{\Psi_2 \varphi\} \in L^1$.

Note that $\Psi_2 \varphi \in H^{s,r}$ for some $r \geqslant 1$ and $s>n/r$. Then the last part in the proof of Lemma 3.4 in \cite{Hu--Salo--Vesalainen} --- which uses a dyadic partition of unity and the H\"older and Hausdorff--Young inequalities --- implies that $\mathscr{F}\{\Psi_2 \varphi\} \in L^1 \cap L^\infty$, from which the claim follows.

Consider the case~\ref{caseHold}) now. According to Triebel \cite{Triebel1}, first theorem in Section 2.8.7, the mapping $f \mapsto \chi_{H_j} f$ is continuous in $H^{s,q}(\R^n)$ when $1 \leqslant q < \infty$, $s \geqslant 0$ and $sq < 1$. We have $q \leqslant 2$, so the previous sentence and the compact support of $\Psi_1$ imply that
\[
  \norm{Q f}_{H^{s,q}(\R^n)} \leqslant C (\norm{f}_{H^{s,q'}(\R^n)} + \norm{\Psi_2 \varphi f}_{H^{s,q}(\R^n)})
\]
for $1/q + 1/q'=1$, $s<1/q$.
By Triebel \cite{Triebel2}, last corollary in Section 4.2.2, $C^\alpha(\R^n)$ is a multiplier for $H^{s,q}(\R^n)$ if $s<\alpha$. So
\[
  \norm{\Psi_2 \varphi f}_{H^{s,q}(\R^n)} \leqslant C \norm{\Psi_2 f}_{H^{s,q}(\R^n)} \leqslant C \norm{f}_{H^{s,q'}(\R^n)}
\]
when $s<\alpha$. The claim follows by combining both of these estimates.
\end{proof}

The following proposition gives existence of the complex geometrical optics solutions in the coordinate patch $U \subseteq \R^n$ where we assume that $V$ has a special form.

\begin{proposition}
\label{CGOinM}
Let $n\in\mathbb Z_+$ with $n\geqslant2$ and let $\mathcal C \subset \R^n$ be an open polyhedral cone.
Consider the potential $V = \chi_{\mathcal C} \varphi$ where $\varphi \colon \R^n\longrightarrow\C$ is compactly supported and in a function space $X$, where
\begin{enumerate}[{case} i)]
	\item $n=2$ and $X = C^\alpha(\R^n)$ for some $\alpha\in\mathbb R_+$, or\label{caseN2}
	\item $n=3$ and $X = C^\alpha(\R^n)$ for some $\alpha\in\left]1/4,\infty\right[$, or\label{caseN3}
	\item $n\geqslant2$ and $X = H^{s,r}(\R^n)$ for some $r\in\left[1,\infty\right[$ and $s\in\left]n/r,\infty\right[$.\label{caseQb}
\end{enumerate}
Let $U \subseteq \R^n$ be a non-empty open set and let $K\colon U\longrightarrow\R_+$ be smooth.

Let $H_K$ be the free operator of \eqref{HKformula}, $\nu \in \{0,1\}$ and $\lambda \in\mathbb R_+$.
Let $B$ be a bounded non-empty open set such that $\overline{B} \subset U$. 
There is a constant $s_0\in\mathbb R_+$ such that the equation
\[
  (H_K + \lambda^\nu\, V - \lambda)u_0 = 0
\]
has a complex geometrical optics solution $u_0\in H^2(B)$,
\[
  u_0(x) = K(x)^{(n-2)/2}\,\exp(\rho\cdot x)\,(1+\psi(x)),
\]
if $\rho\in\C^n$ is a complex vector satisfying $\rho\cdot\rho=0$ and $\abs{\rho} > s_0$.

In the cases~\ref{caseN2}) and~\ref{caseN3}) there is $p \in [6,\infty[$ for which $\psi|_{B} \in L^p(B)$ with
\[
  \norm{\psi}_{L^p(B)} \leqslant C \abs{\rho}^{-n/p-\delta}
\]
for some $\delta = \delta(n,\alpha) \in\mathbb R_+$. In the case~\ref{caseQb}) we may choose any $p\in\left[2,\infty\right[$ and have the estimate
\[
  \norm{\psi}_{L^p(B)} \leqslant C \abs{\rho}^{-1}.
\]
In both cases $\norm{\psi}_{H^2(B)} \leqslant C \abs{\rho}^2$.
\end{proposition}

\begin{proof}
Use the ansatz $u_0(x) = K(x)^{(n-2)/{2}} u(x)$. Lemma~\ref{conjugation-lemma} gives
\begin{align*}
&K^{-(n+2)/{2}}(H_K + \lambda^\nu \,V - \lambda) u_0 \\
	&\qquad= -\Delta u + \frac{(n-2)(n\abs{\nabla K}^2 - 2K\Delta K) - (n-1)^2 + 4\,\lambda^\nu\, V - 4\lambda}{4\,K^2}\, u
\end{align*}
in $U$.

Since $K$ is smooth and positive, and $\overline{B}$ is compact, it has a positive lower bound there and the function $1/K$ is smooth. Let $\Phi \in C^\infty_{\mathrm c}(U)$ be constant $1$ in $B$.
Denote by $Q \colon \R^n \longrightarrow \C$ the extension of
\[
  Q = \Phi\, \frac{(n-2)(n\abs{\nabla K}^2 - 2K\,\Delta K) - (n-1)^2 + 4\,\lambda^\nu\, V - 4\lambda}{4\,K^2}
\]
by zero to $\R^n$. We will build a complex geometrical optics solution to the equation $(-\Delta + Q)u=0$ in $\R^n$ and then restrict $u$ to $B$. In this set it is a solution to the original equation since $\Phi\equiv1$ there.

Let us show that the equation $(-\Delta + Q) u = 0$ has a solution $u(x) = e^{\rho\cdot x}(1 + \psi(x))$ in $\R^n$. Equivalently, let us solve for
\[
  (-\Delta + 2\rho\cdot\nabla + Q)\psi = -Q
\]
in $\R^n$. Note that $Q$ satisfies the requirements of Lemma~\ref{potEstims}.

The case~\ref{caseQb}) follows from \cite{Blasten--Paivarinta--Sylvester}. According to the case~\ref{caseSob}) in Lemma~\ref{potEstims} we have
\[
  Q \in \widehat{B^1_{q,1}} \qquad \text{and} \qquad \norm{ Q g }_{\widehat{B^1_{q,1}}} \leqslant C \norm{g}_{\widehat{B^{-1}_{q,\infty}}}
\]
for any $q \in\left]1,\infty\right[$. Here we used the notation from that same article: $\widehat{B^s_{p,q}} = \mathscr{F} B^s_{p,q}$ are Fourier transforms of Besov spaces with domain $\R^n$. In \cite{Blasten--Paivarinta--Sylvester}, Proposition 4.2, it was showed that the Faddeev operator maps
\[
  (-\Delta + 2\rho\cdot\nabla)^{-1} \colon \widehat{B^1_{q,1}} \longrightarrow \widehat{B^{-1}_{q,\infty}}
\]
with norm estimate $C \abs{\rho}^{-1}$. This implies the existence of a solution $\psi \in \mathscr{F} B^{-1}_{q,\infty}$. But choosing $q = (1-1/p)^{-1} \leqslant 2$, and using the embedding $\mathscr{F} B^{-1}_{q,\infty} \hookrightarrow L^p(B)$ implies that $\psi_{|B} \in L^p(B)$ with norm at most $C \abs{\rho}^{-1}$. 

The cases~\ref{caseN2}) and~\ref{caseN3}) follow from Proposition 3.3 in \cite{Paivarinta--Salo--Vesalainen} and the proof of Theorem 3.1 therein which uses an estimate from \cite{Kenig--Ruiz--Sogge}. More details follow to make the index calculations in the proof of the former clearer. Proposition 3.3 in \cite{Paivarinta--Salo--Vesalainen} says that if $1<q<2$, $1/q+1/q'=1$,
\[
  \frac{1}{q}-\frac{1}{q'} \in \left[ \frac{2}{n+1}, \frac{2}{n} \right[,
\]
$Q$ satisfies \eqref{QSobMult}, and $\abs{\rho}$ is large enough, then there is a solution $\psi \in H^{s,q'}(\R^n)$ which satisfies
\[
  \norm{\psi}_{H^{s,q'}(\R^n)} \leqslant C \abs{\Im\rho}^{n(1/q-1/q')-2} \norm{Q}_{H^{s,q}(\R^n)}.
\]
The slight differences in notation between their estimate and ours follows from our choice of having $\rho\cdot\nabla$ instead of $\rho\cdot D$. Hence the upper bound has $\Im\rho$ instead of $\Re\rho$. Moreover $\sqrt{2}\abs{\Im\rho} = \abs{\rho}$ since $\rho\cdot\rho=0$.

The Sobolev embedding theorem says that $H^{s,q'}(\R^n) \hookrightarrow L^p(\R^n)$ if $0\leqslant s<n/q'$ and then $p$ is determined by $-n/p = s-n/q'$. The claim follows after making sure that all the constraints for the parameters used above are satisfied. Namely
\begin{itemize}
\item having \eqref{QSobMult} requires $1 \leqslant q \leqslant 2$ and
  $0 \leqslant s < \min(1/q, \alpha)$,
\item using Proposition~3.3 in \cite{Paivarinta--Salo--Vesalainen}
  requires $1<q<2$, $1/q+1/q'=1$ and $2/(n+1) \leqslant 1/q-1/q' <
  2/n$,
\item using the Sobolev embedding to $L^p(\R^n)$ requires $0\leqslant
  s < n/q'$ and $p$ is given by $-n/p = s-n/q'$,
\item and finally, for having enough decay in the exponent of
  $\abs{\rho}$ in the final estimate, we require that $n(1/q-1/q')-2 <
  -n/p$
\end{itemize}
We omitted the requirement that $p\geqslant 6$ because it will be implied by these. Checking the conditions becomes easier by writing everything with respect to $s$ and $1/q$. With a little effort one can see that the above are equivalent to
\begin{align*}
  &\frac{n+3}{2n+2} \leqslant \frac{1}{q} < \frac{n+2}{2n}, \qquad
	&&0 \leqslant s < \min\left(\frac{1}{q}, \alpha\right),\\
	&0 \leqslant s < n-\frac{n}{q}, \qquad
	&&\frac{n}{q} - 2 < s,
\end{align*}
after which $q'$ is determined by $1/q+1/q'=1$ and $p$ is given by
$-n/p = s-n/q'$. The conditions can be illustrated as follows, where
we leave out the requirement $s<\alpha$ which will have to be checked
separately.
\begin{center}
\includegraphics{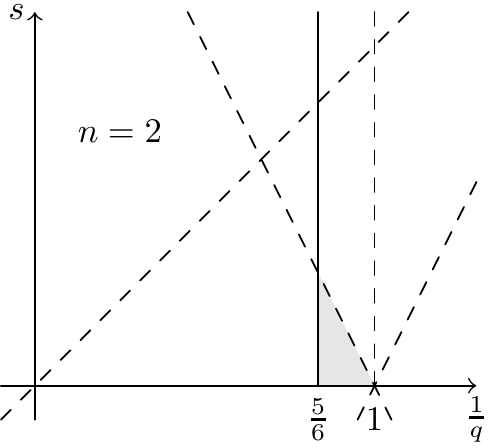} \quad \includegraphics{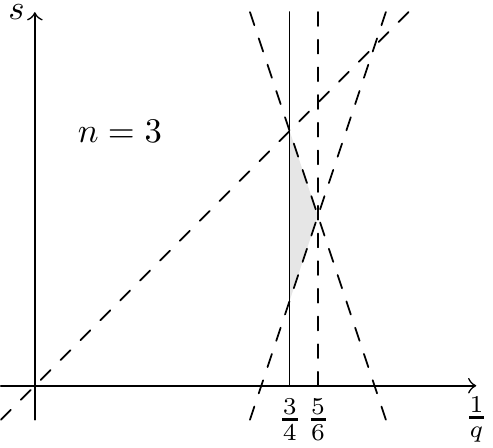}
\end{center}

Adding the condition $s<\alpha$, we see that the above have a solution in 2D if $\alpha>0$, and in 3D if $\alpha>1/4$. In the former case a solution is $(1/q,s)=(5/6,0)$ which gives $1/q' = 1/6$, $n/p = n/q' - s= 1/3$ and $p = (n/p)^{-1}n = 6$. The exponent of $\abs{\rho}$ becomes then
\[
n\left(\frac1q - \frac1{q'}\right) - 2 = 2\left( \frac56 - \frac16
\right) - 2 = - \frac23 = -\frac26 - \frac13 = -\frac{n}{p} - \frac13
\]
which is of the form $-n/p-\delta$, with $\delta=1/3>0$.

In the second case a solution is given by $(1/q,s) = (3/4, s)$ for any $s$ satisfying $1/4 < s < \min(3/4,\alpha)$. Then $1/q' = 1/4$, $n/p = n/q'-s = 3/4-s$ and $p = (n/p)^{-1}n = 12/(3-4s)$, a number between $6$ and infinity depending on $s$. The exponent of $\abs{\rho}$ is
\[
n\left(\frac1q - \frac1{q'}\right) - 2 = 3\left(\frac34 -
\frac14\right)-2 = -\frac12 = s - \frac34 + \frac14 - s = -\frac{n}{p}
+ \frac14-s,
\]
again of the form $-n/p-\delta$, where now $\delta=s-1/4>0$ for any $s$ satisfying the only required condition of $1/4 < s < \min(3/4,\alpha)$. We can choose $s$ based on $\alpha$, for example $s = (1/4+\min(3/4,\alpha))/2$, which suits us well.

We have shown that in all cases there is a solution $\psi \in L^p(B)$
for some $p$, and it has fast enough decay as
$\abs{\rho}\to\infty$. We also note that above it was always possible
to choose $p\geqslant 2$. Hence in all cases we have $\psi \in
L^2(B')$ for any bounded domain $B' \supset \supp \Phi$. Elliptic
interior regularity implies that $\psi_{|B} \in H^2(B)$, and the
required norm estimate for $u_0$ follows.
\end{proof}

We note further that the distance from the line with slope $n$
determines the decay rate of $\norm{\psi}_p$. For the fastest decay
one should have $\alpha \geqslant 1/3$ in 2D and $\alpha \geqslant
3/4$ in 3D.

\section{Corner scattering} \label{cornerScat}

In this section we will prove that corners that are conformal to Euclidean angles in two dimensions or the corner of a Euclidean hypercube in higher dimensions always scatter. The proof proceeds as follows: Assuming the existence of a nonscattering energy --- a nontrivial incident wave such that the corresponding scattering amplitude vanishes --- we use Rellich's theorem to show that the scattered wave is zero outside the support of the potential. Integration by parts gives us a type of orthogonality relation. This is a relation between the potential function, the incident wave and a function which we may choose.

Choosing a family of complex geometrical optics solutions in the orthogonality relation will be useful. The remaining steps involve choosing a suitable set of coordinates for the hyperbolic space under which the corner of interest looks like a Euclidean corner, and then estimating the decay rates of various terms involved in the orthogonality relation. This reduces the problem to showing that the Laplace transform of the product of the characteristic function of a cone and a nontrivial harmonic polynomial cannot vanish identically on the complex characteristic manifold of points $\rho\in\mathbb C^n$ satisfying $\rho\cdot\rho=0$, a problem which has been solved in the papers \cite{Blasten--Paivarinta--Sylvester} and \cite{Paivarinta--Salo--Vesalainen}.

\subsection{An orthogonality relation}

The first step in the proof is to derive a near-orthogonality relation from the hypothesis that a non-scattering energy exists. The following lemma gives that. If the neighbourhood $B$ is large enough to contain the support of the potential $V$ then the right-hand side vanishes, and thus $u_0$ and $w$ are orthogonal, with inner products weighted by $V\,\mathrm d\mu$.

Compared to strict orthogonality, the identity stated below is more useful when the potential lacks smoothness outside of $B$. We will later let $u_0$ be a complex geometrical optics solution whose boundary values decay exponentially as the relevant parameter tends to infinity, so this relation is as useful to us as orthogonality is.
\begin{lemma}
\label{orthogonalityRelationLemma}
Let $n\geqslant2$ be an integer, let $\lambda\in\mathbb R_+$, and let $\Omega$ be a bounded non-empty open set in $\mathbb H^n$ such that the interior of its complement is connected. Furthermore, let $V\in L^\infty(\mathbb H^n)$ vanish outside $\Omega$, and let $v,w\in B^\ast(\mathbb H^n)$ be solutions to the equations
\[\begin{cases}\left(H_0+V-\lambda\right)v=0,\text{\ and}\\
\left(H_0-\lambda\right)w=0,\end{cases}\]
in $\mathbb H^n$ so that $v-w\in\mathring B^\ast(\mathbb H^n)$. 

Let $B \subset \mathbb H^n$ be a smooth bounded domain and let $u_0\in H^2(B)$ be a solution to
\[\left(H_0+V-\lambda\right)u_0=0\]
in $B$. Then $v,w\in H^2_{\mathrm{loc}}(\mathbb H^n)$ and
\begin{equation}\label{H0ortho} \int\limits_{B \cap \Omega}V\,u_0\,w\,\mathrm d\mu = \int\limits_{\partial B} \big( (v-w)\, \partial_\nu u_0 - u_0 \, \partial_\nu(v-w) \big)\, \mathrm d\sigma \end{equation}
with $v-w$ and $\partial_\nu(v-w)$ vanishing identically on $\partial B \setminus \Omega$. 
In the upper half-space coordinates of $\mathbb H^n$ we have $\partial_\nu = x_n\, \partial_N$ where $N(x)$ is Euclidean unit exterior normal vector to $B$ at $x\in\partial B$, and $\mathrm d\mu = x_n^{-n}\,\mathrm dx$, $\mathrm d\sigma = x_n^{n-1}\,\mathrm dS$ where $\mathrm dx$ and $\mathrm dS$ are the Euclidean volume and boundary volume forms of $B$.
\end{lemma}

\begin{proof}
Rellich's classical lemma has a natural analogue in $\mathbb H^n$ (see e.g.\ Theorem 2.10 in \cite{Isozaki--Kurylev} for a more than sufficiently general version). In particular, we have $v\equiv w$ far away, and by unique continuation $v\equiv w$ outside $\Omega$. Elliptic interior regularity estimates imply that $v$ and $w$ are locally in $H^2$ since $V$ is bounded. Hence $v-w$ and $\partial_\nu(v-w)$ exist as $L^2$ functions on $\partial B$ and they both vanish on $\partial B \setminus \Omega$.

Note the integration by parts formula for $H_0$,
\begin{equation}\label{integration-by-parts-for-H0}
  \int\limits_B u\, H_0 v\, \mathrm d\mu = \int\limits_{\partial B} \big( u\, \partial_\nu v - v\, \partial_\nu u \big) \,\mathrm d\sigma + \int\limits_{B} v\, H_0 u\, \mathrm d\mu,
\end{equation}
which follows from
\begin{align*}
  &\int\limits_{B} u \,\bigl(-x_n^2\,\Delta + (n-2)\,x_n \,\partial_n\bigr)\,v \,\frac{\mathrm dx}{x_n^n}\\
  &\qquad = - \int\limits_B u \,\Delta v\, \frac{\mathrm dx}{x_n^{n-2}} + (n-2) \int\limits_B u\, \partial_n v \,\frac{\mathrm dx}{x_n^{n-1}} \\
  &\qquad = -\int\limits_{\partial B} u \,N\cdot\nabla v \,\frac{\mathrm dS}{x_n^{n-2}} + \int\limits_B \nabla v \cdot \nabla \frac{u}{x_n^{n-2}} \,dx + (n-2) \int\limits_B u\, \partial_n v\, \frac{\mathrm dx}{x_n^{n-1}} \\
  &\qquad = -\int\limits_{\partial B} u\, x_n\, N\cdot\nabla v\, \frac{\mathrm dS}{x_n^{n-1}} + \int\limits_B \nabla v\cdot\nabla u\, \frac{\mathrm dx}{x_n^{n-2}} \\
  &\qquad \quad -(n-2)\int\limits_B \nabla v \cdot u\, e_n\, \frac{\mathrm dx}{x_n^{n-1}} + (n-2)\int\limits_B u\, \partial_n v\, \frac{\mathrm dx}{x_n^{n-1}} \\
    &\qquad = -\int\limits_{\partial B} u\, x_n\, N\cdot\nabla v\, \frac{\mathrm dS}{x_n^{n-1}} + \int\limits_B \nabla v\cdot\nabla u\, \frac{\mathrm dx}{x_n^{n-2}},
    \end{align*} 
where the two terms with $n-2$ simply cancel out, and where the last integral is symmetric with respect to $u$ and $v$, so that we obtain
\begin{align*}
&\int\limits_{B} u \,\bigl(-x_n^2\,\Delta + (n-2)\,x_n \,\partial_n\bigr)\,v \,\frac{\mathrm dx}{x_n^n}+\int\limits_{\partial B} u\, x_n\, N\cdot\nabla v\, \frac{\mathrm dS}{x_n^{n-1}} \\
&\qquad=\int\limits_B \nabla v\cdot\nabla u\, \frac{\mathrm dx}{x_n^{n-2}}
=\int\limits_B \nabla u\cdot\nabla v\, \frac{\mathrm dx}{x_n^{n-2}}\\
&\qquad=\int\limits_{B} v \,\bigl(-x_n^2\,\Delta + (n-2)\,x_n \,\partial_n\bigr)\,u \,\frac{\mathrm dx}{x_n^n}+\int\limits_{\partial B} v\, x_n\, N\cdot\nabla u\, \frac{\mathrm dS}{x_n^{n-1}} .
\end{align*}
Using \eqref{integration-by-parts-for-H0}, we get
\begin{align*}
0&=\int\limits_{B} \bigl(\left(H_0+V-\lambda\right)u_0\bigr)  \left(v-w\right) \,\mathrm d\mu \\
& = \int\limits_{\partial B} \big( (v-w)\, \partial_\nu u_0 - u_0 \, \partial_\nu (v-w) \big)\, \mathrm d\sigma \\
& \qquad + \int\limits_{B} u_0\, \left(H_0+V-\lambda\right)\left(v-w\right) \,\mathrm d\mu \\
& = \int\limits_{\partial B} \big( (v-w)\, \partial_\nu u_0 - u_0 \, \partial_\nu (v-w) \big) \,\mathrm d\sigma - \int\limits_{B\cap\Omega} u_0\, V\, w \,\mathrm d\mu
\end{align*}
since $V\equiv0$ outside of $\Omega$. 
\end{proof}

\begin{corollary}\label{orthoInCoordsCorol}
Let $u_0$, $v$, $w$, $V$, $B$ and $\Omega$ be as in Lemma~\ref{orthogonalityRelationLemma} with variable denoted by $y\in\mathbb H^n$ instead of $x$. Let $U\subseteq \R^n$ be open, $V\subset\mathbb H^n$ be open with $B\subset V$, and $x:V\longrightarrow U$ a diffeomorphism. Assume that under these coordinates the pushforward of the hyperbolic metric at a point $y\in V$ is given at $x=x(y)\in U$ by the Riemannian metric
\[
  g_x(a,b) = \frac{1}{(K(x))^2} \sum_{i=1}^n a_i b_i
\]
for two vectors $a,b$ in the tangent space to $U$ at $x$ where $a_i,b_i$ are their orthonormal coordinates. Here moreover $K(x)\colon U\longrightarrow\mathbb R_+$ is assumed smooth. Then the hyperbolic volume form $\mathrm d\mu$ is $K^{-n}\, \mathrm dx$ in these coordinates, where $\mathrm dx$ is the Euclidean volume form. Moreover the identity \eqref{H0ortho} becomes
\begin{equation} \label{orthoInCoords}
  \int\limits_{\Omega \cap B}V\, u_0\,w\, \frac{\mathrm dx}{K^n} = \int\limits_{\partial B} \big( (v-w)\, \partial_N u_0 - u_0\, \partial_N (v-w) \big) \,\frac{\mathrm d S}{K^{n-2}}
\end{equation}
in these coordinates. The differences $v-w$ and $\partial_\nu(v-w)$ vanish identically on $\partial B \setminus \Omega$.
Here $N$ and $\mathrm d S$ are the Euclidean exterior unit normal vector and boundary measure on $\partial B$, respectively, and the pushforward versions of $B, \Omega \subset \mathbb H^n$ into $U$ are denoted with the same symbols.
\end{corollary}

\begin{proof}
The given form of the Riemannian metric implies that
\[
  \mathrm d\mu = K^{-n}\, \mathrm dx, \quad \mathrm d\sigma = K^{-(n-1)} \,\mathrm dS, \quad \frac{\partial}{\partial \nu} = \frac{\partial}{K^{-1}\,\partial N}
\]
i.e. the hyperbolic volume of an infinitesimal set in $U$ is its Euclidean volume multiplied by $K^{-n}$. For boundary sets multiply by $K^{-(n-1)}$, and for lengths and vectors by $K^{-1}$. Then apply these to \eqref{H0ortho}.
\end{proof}

\subsection{From CGO solutions to Laplace transforms}
In this section we will prove a lemma that will bring together all the major players in corner scattering: complex geometrical optics solutions, the non-scattering wave, and the shape of the corner. This is the argument from \cite{Blasten--Paivarinta--Sylvester, Paivarinta--Salo--Vesalainen}. We start by stating what we mean by a function having a specified order at a point:

\begin{definition}
Let $\Omega \subseteq \R^n$ be open and let $x_0 \in \Omega$. Then $f \in L^1_{\mathrm{loc}}(\Omega)$ has \emph{order} $N\in\mathbb Z\cup\left\{+\infty,-\infty\right\}$ at $x_0$ if
\[
N = \sup\{ M \in \mathbb{Z} \mid \exists C_M \in\mathbb R_+: \text{$\abs{f(x)} \leqslant C_M \abs{x-x_0}^M$ for a.e.\ $x\in\Omega$ near $x_0$}\}.
\]
\end{definition}

\begin{lemma}\label{orderSplitLemma}
Let $\Omega \subseteq \R^n$ be open, let $x_0\in\Omega$ and let $f\colon\Omega\longrightarrow\C$ be smooth. Assume that $N\in\mathbb Z\cup\left\{\pm\infty\right\}$ is the order of $f$ at $x_0$. Then $N \geqslant 0$ and $\partial^\alpha f(x_0)=0$ whenever $\abs{\alpha}<N$. Also, if $M\in\mathbb Z_+$ is such that $\partial^\alpha f(x_0)=0$ for $\abs{\alpha}<M$, then $N \geqslant M$.

Moreover if $N<\infty$ then there is a homogeneous polynomial $P_N$ of degree $N$ with complex coefficients and $C\in\mathbb R_+$ such that
\[
\abs{f(x) - P_N(x-x_0)} \leqslant C \abs{x}^{N+1}
\]
for $x\in\Omega$ in a neighbourhood of $x_0$.
\end{lemma}
\begin{proof}
Since $f$ is smooth, it is bounded in a neighbourhood of $x_0$, and so $N \geqslant 0$. Now, by Taylor's theorem, for any $\widetilde N\in\mathbb Z_+$,
\[
f(x) = \sum_{\lvert \alpha \rvert \leqslant\widetilde N} \frac{\partial^\alpha f(x_0)}{\alpha!}(x-x_0)^\alpha + \sum_{\lvert\alpha\rvert =\widetilde N} h_\alpha(x)(x-x_0)^\alpha
\]
and $h_\alpha(x)\longrightarrow0$ as $x\longrightarrow x_0$. Since $f$ is smooth so are the $h_\alpha$. The mean value theorem implies that $\abs{h_\alpha(x)} \leqslant C\abs{x-x_0}$ in a neighbourhood of $x_0$.

The claims about the derivatives follows directly. If $N<\infty$ we can choose $\widetilde N\geqslant N$ and $P_N$ to be the sum of the terms of order $N$ in the first sum above.
\end{proof}

The following lemma lacks the factor $1/K^n$ from the hyperbolic metric. This will not be an issue since $K$ will be smooth, positive and bounded from below on compact sets in $U$, and so we can simply let the function $\varphi$ hold this factor.

\begin{lemma}\label{orig90DegCornerScatter}
Let $\Omega \subset \R^n$ be a non-empty open set and let $x_0 \in \partial\Omega$. Assume that there is an open cone $\mathcal C$ with vertex $x_0$ such that for some $h\in\mathbb R_+$ we have $\Omega \cap B(x_0,h) = \mathcal C \cap B(x_0,h)$.
Let $V = \chi_{\mathcal C} \varphi$ in $B(x_0,h)$, with $\varphi\in C^{\alpha}(\overline{B}(x_0,h))$ for some $\alpha\in\mathbb R_+$, and $\chi_{\mathcal C}$ the characteristic function of the cone $\mathcal C$.

Assume that $\rho_0 \in \mathbb{C}^n$ is such that $\rho_0\cdot\rho_0 = 0$, $\abs{\rho_0}=1$ and that the function $\exp(\rho_0\cdot (x-x_0))$ is integrable in $\mathcal C$. 
Let there be $C\in\mathbb R_+$, $p\in\left[1,\infty\right[$, and a sequence $s_1$, $s_2$, \dots of positive real numbers satisfying $s_j \longrightarrow +\infty$ as $j\longrightarrow\infty$, and functions $u_0^1,u_0^2,\ldots \in L^p(B(x_0,h))$ satisfying
\[
  u_0^j(x) = e^{\rho_j\cdot (x-x_0)}(1+\psi_j(x)),
\]
for every $j\in\mathbb Z_+$, where $\rho_j = s_j \rho_0$ and $\norm{\psi_j}_{L^p(B(x_0,h))} \leqslant C s_j^{-n/p-\delta}$ for some $\delta\in\mathbb R_+$.

Let $w$ be smooth in $B(x_0,h)$ and of order at least $N \in \mathbb Z_+\cup\left\{0\right\}$ at $x_0$. Then we may write
\[
  w(x) = P_N(x-x_0) + r_{N+1}(x),
\]
where $P_N$ either vanishes identically or is a homogeneous complex polynomial of degree $N$, and where the error term satisfies $\left|r_{N+1}(x)\right|\leqslant C\left|x-x_0\right|^{N+1}$ for $x$ near $x_0$.

As a consequence, if
\begin{equation} \label{orthoIntegral}
  \int\limits_{\Omega \cap B(x_0,h)}V\,u_0^j\,w\, \mathrm dx=o(s_j^{-N-n})
\end{equation}
as $j\longrightarrow\infty$, then we have
\[
  \varphi(x_0) \int_{(-x_0)+\mathcal C} e^{\rho_0\cdot x}\, P_N(x) \,\mathrm dx = 0,
\]
i.e. the Laplace transform of $P_N$ over $\mathcal C$ translated to the origin vanishes at $\rho_0$ if $\varphi(x_0)\neq0$.
\end{lemma}
\begin{proof}
By translating all the sets and functions we may assume that $x_0 = 0$. This simplifies notation. Also, in the following, the constant factors $C$ in estimates are allowed to depend on $n$, $\Omega$, $h$, $\mathcal C$, $\varphi$, $\alpha$, $N$, $w$, $p$ and $\rho_0$, but not on $j$, the point being that in the end we will take $j\longrightarrow\infty$.

According to Lemma~\ref{orderSplitLemma} we get the splitting $w = P_N + r_{N+1}$, with $P_N$ of order $N$ and $\abs{r_{N+1}(x)} \leqslant C \abs{x}^{N+1}$ in a neighbourhood of $0$. 
We will let $B = B(0,\varepsilon) \subset B(0,h)$ be such a neighbourhood. We can assume that $h = \varepsilon$, because \eqref{orthoIntegral} decays sufficiently fast also when $h=\varepsilon$. This follows since $\exp(s_j\rho_0\cdot x)$ decays exponentially in $\Omega \cap B(x_0,h) \setminus B(x_0,\varepsilon)$ as $j\longrightarrow\infty$ by the assumption on $\rho_0$.

We will split the integral in \eqref{orthoIntegral} given in the statement by splitting each of the factors of the integrand into a main term and a ``higher order'' term. In particular,
\begin{align*}
w(x) &= P_N(x) + r_{N+1}(x), \\
u_0^j(x) &= e^{\rho_j\cdot x}(1 + \psi_j(x)), \\
V(x) &= \chi_{\mathcal C} \big(\varphi(0) + (\varphi(x)-\varphi(0)) \big),
\end{align*}
with the estimates
\begin{align*}
\abs{r_{N+1}(x)} &\leqslant C \abs{x}^{N+1},\\
\norm{\psi_j}_{L^p(B(0,h))} &\leqslant C s_j^{-n/p-\delta},\\
\abs{\varphi(x)-\varphi(0)} &\leqslant \norm{\varphi}_{C^\alpha(\overline{B}(0,h))} \abs{x}^\alpha,
\end{align*}
for $x\in B(0,h)$.

Now, writing $B = B(0,h)$ and letting $\psi=\psi_j$ and $\rho = s\rho_0$ with $s=s_j$,  $j \in \{1, 2, \ldots\}$, we have
\begin{align}
\label{monsterMRn}
&\int\limits_{\Omega \cap B} V(x)\, u_0^j(x)\, w(x)\, \mathrm dx \notag\\
&\qquad = \int\limits_{\mathcal C \cap B} \big(\varphi(0) + (\varphi(x)-\varphi(0))\big)\, e^{\rho\cdot x}\,(1+\psi(x)) \,(P_N(x) + r_{N+1}(x)) \,\mathrm dx \notag\\
&\qquad = \int\limits_{\mathcal C \cap B} e^{\rho\cdot x} \,\varphi(0)\, P_N(x) \,\mathrm dx + \int\limits_{\mathcal C \cap B} e^{\rho\cdot x}\, (\varphi(x)-\varphi(0))\, P_N(x)\, \mathrm dx \notag\\
&\qquad \quad + \int\limits_{\mathcal C \cap B} e^{\rho\cdot x}\,\varphi(x) \,r_{N+1}(x)\, \mathrm dx + \int\limits_{\mathcal C \cap B} e^{\rho\cdot x} \,\varphi(x) \,w(x) \,\psi(x)\, \mathrm dx \notag\\
&\qquad = \varphi(0) \int\limits_{\mathcal C} e^{\rho\cdot x}\,P_N(x)\, \mathrm dx - \varphi(0) \int\limits_{\mathcal C \setminus B} e^{\rho\cdot x}\, P_N(x)\, \mathrm dx \notag\\
&\qquad\quad + \int\limits_{\mathcal C \cap B} e^{\rho\cdot x} \,(\varphi(x)-\varphi(0))\, P_N(x) \,\mathrm dx + \int\limits_{\mathcal C \cap B} e^{\rho\cdot x}\,\varphi(x) \,r_{N+1}(x) \,\mathrm dx \notag\\
&\qquad\quad + \int\limits_{\mathcal C \cap B} e^{\rho\cdot x} \,\varphi(x)\, w(x) \,\psi(x)\, \mathrm dx.
\end{align}

Let us consider the individual integrals next. First,
\begin{equation} \label{intM1Rn}
\abs{ \int_{\mathcal C \setminus B} e^{\rho\cdot x}\, P_N(x) \,\mathrm dx } \leqslant \int_{\mathcal C \setminus B} e^{s\Re\rho_0\cdot x}\, \abs{P_N(x)} \,\mathrm dx \leqslant C e^{-Cs}
\end{equation}
because we assumed that $\exp(\rho_0\cdot x)$ is integrable in $\mathcal C$, so the integral decays exponentially in $\mathcal C \setminus B$ as $j\longrightarrow\infty$. The three integrals over $\mathcal C \cap B$ are dealt with H\"older's inequality. We start with
\begin{align} \label{intM2Rn}
&\abs{\int_{\mathcal C \cap B} e^{\rho\cdot x}\, (\varphi(x)-\varphi(0))\, P_N(x) \,\mathrm dx} \leqslant C \int_{\mathcal C} e^{s\Re\rho_0\cdot x}\, \abs{x}^{N+\alpha}\, \mathrm dx \notag\\
&\qquad = C\, s^{-N-n-\alpha} \int_{\mathcal C} e^{\Re\rho_0 \cdot y}\, \abs{y}^{N+\alpha}\, \mathrm dy,
\end{align}
and the last integral is finite. Next, 
\begin{align} \label{intM3Rn}
&\abs{\int_{\mathcal C \cap B} e^{\rho\cdot x}\,\varphi(x)\, r_{N+1}(x) \,\mathrm dx} \leqslant C \int_{\mathcal C} e^{s\Re\rho_0\cdot x}\, \abs{x}^{N+1}\, \mathrm dx \notag\\
&\qquad = C\, s^{-N-n-1} \int_{\mathcal C} e^{\Re\rho_0\cdot y}\, \abs{y}^{N+1} \,\mathrm dx,
\end{align}
and again the dependence on $s$ is made explicit, this time by taking the $L^\infty$-norm of $\varphi$. For the last integral note that $\abs{w(x)} \leqslant C\abs{x}^N$ since it is of order at least $N$. Let $p'$ be the dual exponent of $p$, so that $1=1/p+1/{p'}$. Then
\begin{align} \label{intM4Rn}
&\abs{\int_{\mathcal C \cap B} e^{\rho\cdot x}\, \varphi(x)\, w(x)\, \psi(x)\, \mathrm dx} \leqslant C \norm{ e^{s\rho_0\cdot x}\, w(x) }_{L^{p'}(\mathcal C \cap B)} \norm{\psi}_{L^p(B)} \notag\\
&\qquad \leqslant C\, s^{-N-n/p'}\, s^{-n/p-\delta} = C \,s^{-N-n-\delta}
\end{align}
since $\abs{\exp(s\rho_0\cdot x)w(x)} \leqslant C \exp(s\Re\rho_0\cdot x) \,\abs{x}^N$ and by the scaling properties of the $L^{p'}$-norm.

Before plugging estimates \eqref{intM1Rn}, \eqref{intM2Rn}, \eqref{intM3Rn} and \eqref{intM4Rn} into \eqref{monsterMRn} we will change variables $y = sx$ in the first term in the right-hand side of \eqref{monsterMRn}. The decay of the left-hand side and the estimates of the individual integrals above will leave us with
\[
s^{-N-n} \abs{\varphi(0) \int_{\mathcal C} e^{\rho_0\cdot y}\, P_N(y)\, \mathrm dy } =o(s^{-N-n}).
\]
The claim follows by letting $s=s_j\longrightarrow\infty$.
\end{proof}

\subsection{Finishing the proofs}

We will first show that if $w$ satisfies the free equation $(H_0 - \lambda)\,w = 0$ with $\lambda\in\mathbb R_+$ in the hyperbolic space $\mathbb H^n$, then the principal term of its Taylor expansion is harmonic in the chosen coordinates. This is true for any coordinates that transform $H_0$ into $H_K$. After that we can prove that certain types of corners always scatter in the hyperbolic space.

\begin{lemma}
\label{harmonicSplitM}
Let $U \subseteq \R^n$ be a non-empty open set and let $K\colon U\longrightarrow\R_+$ be smooth.
Let $w \in L^1_{\mathrm{loc}}(U)$ and assume that $w$ does not vanish almost everywhere and that $(H_K - \lambda)\, w = 0$ in $U$, where $H_K$ is given by \eqref{HKformula} on page \pageref{HKformula}, and $\lambda\in\mathbb R$. Then $w$ is smooth and of finite order $N \in \mathbb Z_+\cup\left\{0\right\}$ at any $x_0 \in U$. Moreover, if $P_N$ is the sum of the lowest order terms in the Taylor-expansion of $w$ at $x_0$, given by Lemma~\ref{orderSplitLemma}, then $P_N$ is a harmonic polynomial in $\mathbb R^n$, i.e.\ $\Delta P_N = 0$.
\end{lemma}
\begin{proof}
The function $w$ is smooth because the coefficients of $H_K$ are smooth. Also, $w$ has order $N \in \mathbb Z_+ \cup \{0,+\infty\}$ at $x_0$ by Lemma~\ref{orderSplitLemma}. If $N = \infty$ then $w$ vanishes to infinite order at $x_0$, and hence would vanish identically in $U$. This follows from results in \cite{Cordes} according to \cite{Baouendi--Zachmanoglou}. Hence $0 \leqslant N < \infty$, and we may assume that $N \geqslant 2$.

We will emphasize the order of a function by writing it as a subscript. Then, by Lemma~\ref{orderSplitLemma} and the differentiability of $K$ at $x_0$, we have
\begin{align*}
w(x) &= P_N(x-x_0) + r_{N+1}(x), &\abs{r_{N+1}(x)} &\leqslant C \abs{x-x_0}^{N+1},\\
K(x) &= K(x_0) + k_1(x), &\abs{k_1(x)} &\leqslant C\abs{x-x_0},
\end{align*}
in a neighbourhood of $x_0$.

According to \eqref{HKformula},
\[
  0 = (H_K - \lambda) w = \big( -K^2 \,\Delta + \left(n-2\right) K \,\nabla K \cdot \nabla - \lambda - (n-1)^2/4 \big) w.
\]
Now
\begin{multline*}
K(x_0)^2\, \Delta P_N(x-x_0) 
= \left(n-2\right) K\, \nabla K \cdot \nabla w - \left(\lambda_0+\left(n-1\right)^2/4\right) w 
\\ - K^2\, \Delta r_{N+1} - 2\,K(x_0) \,k_1\, \Delta P_N(x-x_0) - k_1^2 \,\Delta P_N(x-x_0).
\end{multline*}
Next we use the boundedness of $K$ and $\nabla K$ near $x_0$ to get
\begin{align*}
&\abs{ K(x_0)^2\, \Delta P_N(x-x_0) }
	\leqslant C \,\big( \abs{\nabla w(x)} + \abs{w(x)} \\
&\qquad\qquad		+ \abs{\Delta r_{N+1}(x)} 
		+ \abs{k_1(x)} \cdot\abs{ \Delta P_N(x-x_0)} 
		+ \abs{k_1(x)}^2 \abs{\Delta P_N(x-x_0) } \big).
\end{align*}

Both $w$ and $r_{N+1}$ are smooth and of orders $N$ and at least $N+1$, respectively. By looking at how many derivatives vanish at $x_0$, Lemma~\ref{orderSplitLemma} shows that $\nabla w$ and $\Delta r_{N+1}$ are of order at least $N-1$, and so $\abs{\nabla w(x)}+\abs{\Delta r_{N+1}(x)} \leqslant C \abs{x-x_0}^{N-1}$. Also, $\Delta P_N$ is of degree at most $N-2$, and so using all of the previous estimates we see that
\[
\abs{ K(x_0)^2\, \Delta P_N(x-x_0) } \leqslant C \abs{x-x_0}^{N-1}
\]
in a neighbourhood of $x_0$. Since $\Delta P_N$ either vanishes identically or is a homogeneous polynomial of degree $N-2$, and since $K(x_0) \neq 0$, we have $\Delta P_N = 0$.
\end{proof}

\bigskip
We can now prove a proposition from which the main theorems on non-scattering energies follow easily. The proposition deals with corner scattering in $\R^n$ for the partial differential operator $H_K$ given by \eqref{HKformula} on page \pageref{HKformula}. After this, the strategy for proving the main theorem is to choose a coordinate patch $(U,x)$ for $\mathbb H^n$, and this will fix the function $K$ appearing in the coefficients of $H_K$. The proposition gives conditions, in terms of the coordinate patch, under which the potential will always scatter. These will then have to be translated back into the language of the hyperbolic space.

\begin{proposition}\label{finalProp}
Let $n \geqslant 2$ be an integer and let $U \subseteq \R^n$ be a non-empty open set with a smooth function $K\colon U\longrightarrow\R_+$. Let $\Omega \subset U$ be a bounded non-empty open set such that $\overline{\Omega} \subset U$ and $U \setminus \overline{\Omega}$ has connected interior. Also, let $\mathcal C \subset \R^n$ be an open cone with vertex $x_0 \in U$ and assume that there is $h\in\mathbb R_+$ such that $B(x_0,h) \cap \Omega = B(x_0,h) \cap \mathcal C$ and $h < d(x_0, \partial U)$.

Let $V \in L^\infty(U)$ vanish outside of $\Omega$ and assume that there is a compactly supported function $\varphi\colon\R^n\longrightarrow\C$ in a function space $X$, defined below, such that $V = \chi_{\mathcal C}\, \varphi$ in $B(x_0,h)$. Let $\lambda\in\R_+$. If
\begin{equation}\label{requiredIntByParts}
  \int\limits_{\Omega \cap B(x_0,h)} V\, u_0\, w\, \frac{\mathrm dx}{K^n} = \int\limits_{\partial B(x_0,h)} (v_s\, \partial_N u_0 - u_0\, \partial_N v_s)\, \frac{\mathrm dS}{K^{n-2}}
\end{equation}
for any $u_0 \in H^2(B(x_0,h))$ satisfying $(H_K + V - \lambda)u_0=0$ with $H_K$ given in \eqref{HKformula} and some functions
\begin{itemize}
\item $w \in L^1_{\mathrm{loc}}(U)$ satisfying $(H_K - \lambda)\,w = 0$ and not vanishing almost everywhere, and
\item $v_s \in H^2(B(x_0,h))$ vanishing outside of $\Omega$,
\end{itemize}
then, in
\begin{enumerate}[{case} i)]
\item $n=2$, $X = C^\alpha$ for some $\alpha\in\mathbb R_+$, and $\mathcal C$ has opening angle in ${]{0,\pi}[}$, or\label{finalPropC1}
\item $n=3$, $X = C^\alpha$ for some $\alpha\in\left]1/4,\infty\right[$, and $\mathcal C = {]{0,\infty}[}^3 + x_0$, or\label{finalPropC2}
\item $n\in\left\{2,3,\ldots\right\}$, $X = H^{s,r}$ for some $r \in\left[ 1,\infty\right[$, $s \in\left] n/r,\infty\right[$, and $\mathcal C = {]{0,\infty}[}^n + x_0$,\label{finalPropC3}
\end{enumerate}
we have $\varphi(x_0)=0$.
\end{proposition}

\begin{proof}
We will use Lemma~\ref{orig90DegCornerScatter} with the potential $V'=V \,K^{-n} = \chi_{\mathcal C}\, \varphi\, K^{-n}$. There is no smoothness issue since $K^{-n} \,\varphi \in C^\alpha$ for some $\alpha > 0$ in all cases. This follows from the smoothness and lower bound of $K$ in the compact set $\overline{\Omega}$, and Sobolev embedding when $\varphi \in H^{s,r}$ with $s > n/r$. Let us show that we have functions $u_0$ and $w$ satisfying the lemma's assumptions, and that $\int V \,K^{-n}\, u_0\, w\, \mathrm dx$ decays exponentially when the the complex geometric optics solution parameter tends to infinity.

In all three cases we see that $\mathcal C$ is an open polyhedral cone. Hence Proposition~\ref{CGOinM} implies the existence of a set of complex geometrical optics solutions $u_0^j \in H^2(B(x_0,h))$, where $j$ ranges over $\mathbb Z_+$, and $p \in\left] 1,\infty\right[$ such that
\[
  u_0^j(x) = K(x)^{(n-2)/2}\, e^{\rho_j\cdot(x-x_0)}\,(1+\psi_j(x)), \qquad \norm{\psi_j}_{L^p(B(x_0,h))} \leqslant C\, s_j^{-n/p-\delta}
\]
for every $j\in\mathbb Z_+$, for some $\delta \in\mathbb R_+$, and where $\rho_j = s_j \rho_0 \in\C^n$ has absolute value $s_j$ which tends to infinity as $j \longrightarrow \infty$ and $\rho_0\cdot\rho_0=0$. Note that this form of the bound for $\psi_j$ holds also in case~\ref{finalPropC3} because $p$ can be large, for example $p>n$. We may choose $\rho_0\in\C^n$ such that $\exp(\rho_0\cdot(x-x_0))$ is integrable in a neighbourhood of $\mathcal C$, i.e.\ that $\Re \rho_0 \cdot (x-x_0) \leqslant -\gamma \abs{x-x_0}$ for some $\gamma\in\mathbb R_+$ when $x\in \mathcal C$.

Consider the function $w$ now. By Lemma~\ref{harmonicSplitM} it is smooth, has a finite order $N \in \mathbb Z_+\cup\left\{0\right\}$ at $x_0$, and its lowest order approximation $P_N$ is a harmonic homogeneous polynomial.

Now the only thing left to show for Lemma~\ref{orig90DegCornerScatter} is the exponential decay of $\int V\, K^{-n}\, u_0^j\, w\, \mathrm dx$ as $s_j \longrightarrow \infty$. Recall that $v_s$ vanishes outside $\mathcal C$ in $B(x_0,h)$, which implies the same for its trace and normal derivative on $\partial B(x_0,h)$. On the other hand, by the choice of $\rho_0$, we have $\abs{\exp(\rho_j\cdot(x-x_0))} \leqslant \exp(-\gamma s_j h)$ when $x \in \partial B(x_0,h) \cap \mathcal C$. These imply the decay of the boundary term in the proposition statement, and hence also of the volume integral. Lemma~\ref{orig90DegCornerScatter} implies then that
\[
  \frac{\varphi(x_0)}{K^n(x_0)} \int\limits_{\mathcal C - x_0} e^{\rho_0\cdot x}\, P_N(x) \,\mathrm dx = 0.
\]

We assumed that $P_N$ is a nontrivial polynomial. Then standard arguments in corner scattering imply that $\varphi(x_0) = 0$. In more detail, the arguments of Section 5 in \cite{Paivarinta--Salo--Vesalainen} apply verbatim in the two dimensional case, and show that the Laplace transform cannot vanish for all admissible $\rho_0$. Similarly, in the three and higher dimensional cases where $\mathcal C$ is a right-angled corner, Theorem 2.5 of \cite{Blasten--Paivarinta--Sylvester} imply the same statement. Hence $\varphi(x_0) = 0$.
\end{proof}

\begin{proof}[Proof of Theorem~\ref{2Dthm}]
Use the Poincar\'e disc coordinates for $\mathbb H^2$. Choose these coordinates such that $\mathbb H^2$ is pushforwarded to $U=B(0,1)\subset\R^2$ and $x_0$ has coordinates $0$. Then $\mathcal C$ becomes an Euclidean cone restricted to the disc. By restricting to a smaller neighbourhood we may assume that the pushforward of $B$ is a Euclidean ball $B(0,h)$, $0<h<1$.

Assume that there would be a non-scattering incident wave $w$ of energy $\lambda\in\mathbb R_+$.  Then Lemma~\ref{orthogonalityRelationLemma} implies the following: for any $u_0 \in H^2(B)$ solving $(H_0 + V -\lambda)\,u_0 = 0$ we have
\[
  \int\limits_{B\cap\Omega} V\, u_0\, w\, \mathrm d\mu = \int\limits_{\partial B} ( v_s\, \partial_\nu u_0 - u_0 \,\partial_\nu v_s)\, \mathrm d\sigma
\]
where $v_s \in \mathring B^\ast(\mathbb H^2)$ is the corresponding scattered wave with vanishing far-field pattern. It also vanishes outside $\Omega$. Using the coordinates of the Poincar\'e disc, Corollary~\ref{orthoInCoordsCorol} implies that \eqref{requiredIntByParts} holds with $K(x) = 2/(1-\abs{x}^2)$. Case~\ref{finalPropC1}) holds, so Proposition~\ref{finalProp} implies that $\varphi(x_0)=0$. The contradiction implies that there is no non-scattering incident wave.
\end{proof}

\begin{proof}[Proof of Theorem~\ref{3Dthm}]
Model $\mathbb H^n$ by the $n$-dimensional Poincar\'e ball such that $\mathbb H^n$ is pushed forward to $U=B(0,1)\subset\R^n$ and $x_0$ to the origin $0\in U$. After rotation $\mathcal C$ becomes ${]{0,\infty}[}^n \cap U$. The rest of the proof is verbatim to the proof of Theorem~\ref{2Dthm} except that in addition to case~\ref{finalPropC1}) we also use case~\ref{finalPropC2}) and case~\ref{finalPropC3}) in Proposition~\ref{finalProp}.
\end{proof}

\begin{proof}[Proof of Theorem~\ref{confThm}]
Let $g$ denote the hyperbolic metric in $B\subset\mathbb H^n$ and $G$ the Euclidean one in $\Phi(B)\subset\R^n$, respectively. Let $g' = \Phi_\ast g$ be the pushforward of $g$ to $\Phi(B)$. We may assume that $B$ is small enough guaranteeing that $\Phi$ is a diffeomorphism and $\Phi\circ\Phi^{-1}$ is conformal in $\Phi(B)$. This implies that there is a positive function $K\colon\Phi(B)\longrightarrow\R_+$ such that
\[
  g'_x(a,b) = \frac{1}{K^2(x)}\, G_x(a,b)
\]
for vectors $a, b \in \R^n$ at $x \in \Phi(B)$. We shall pass to an even smaller, smooth neighbourhood $B$ of $x_0$ which maps to a Euclidean ball $B(0,h)$ with $h\in\mathbb R_+$. This allows us to assume that $\inf_{\Phi(B)} K > 0$, $B$ is smooth and that $\Phi(B) = B(0,h)$.

Assume that $H_0 + V$ would have a non-scattering energy $\lambda$ with corresponding incident wave $w$. By Lemma~\ref{orthogonalityRelationLemma} we see that for any $u_0 \in H^2(B)$ solving $(H_0 + V -\lambda)\,u_0=0$ we get
\[
  \int\limits_{B\cap\mathcal C} V\, u_0\, w\, \mathrm d\mu = \int\limits_{\partial B} (v_s \,\partial_\nu u_0 - u_0\, \partial_\nu v_s) \,\mathrm d\sigma
\]
where $v_s$ is the scattered wave corresponding to $w$, and it vanishes outside $\Omega$.

Use Corollary~\ref{orthoInCoordsCorol} for $B$. If we write $V' = V \circ \Phi^{-1}$, $u_0' = u_0 \circ \Phi^{-1}$, $w' = w \circ \Phi^{-1}$ and $v_s' = v_s \circ \Phi^{-1}$, then the Corollary implies
\[
  \int\limits_{\Phi(B)\cap \Phi(\mathcal C)} V'\, u_0'\, w'\, \frac{\mathrm dy}{K^n} = \int\limits_{\partial \Phi(B)} \big(v_s' \,\partial_N u_0' - u_0'\, \partial_N v_s')\big) \,\frac{\mathrm dS}{K^{n-2}}
\]
and $K$ is smooth and bounded below by a positive constant in $\Phi(B)$.

Let $\Delta_g$ denote the Laplace--Beltrami operator in $B$, and so $H_0 = -\Delta_g - \left(n-1\right)^2/4$ by \eqref{H0def} on page \pageref{H0def} where the upper half-space coordinates are used, or by the discussion at the beginning of Section~\ref{CGO}. The pushforward of the Laplacian is
\[\Phi_\ast(\Delta_g \,f) = \Delta_{g'} \,\Phi_\ast f= K^2\, \Delta_G \,\Phi_\ast f- \left(n-2\right) K\, \nabla_G \,K \cdot \nabla_G \,\Phi_\ast f,\]
 where $\Delta_G$ and $\nabla_G$ are the Euclidean Laplacian and gradient in $\Phi(B)$. In other words
\[
  H_0\, f(y) = -K^2\, \Delta_G\, \tilde f(x) + \left(n-2\right)\, K\, \nabla_G\, K \cdot \nabla_G\, \tilde f(x) - \frac{(n-1)^2}{4}\, \tilde f(x)
\]
where $\tilde f(x) := f(\Phi^{-1}(x)) = f(y)$. This is equal to $H_K\, \tilde f(x)$ where $H_K$ is given by \eqref{HKformula} on page \pageref{HKformula}. Moreover the pushforward of the potential $V$ restricted to $B(0,h)$ vanishes outside of a convex cone in two dimensions or ${]{0,\infty}[}^n$ in higher dimensions. The rest of the proof is as in the proofs of Theorems \ref{2Dthm} and \ref{3Dthm}.
\end{proof}

\section{Transmission eigenvalues via quadratic forms} \label{ITE}


Finally, let us focus on transmission eigenvalues and prove Theorems \ref{schrodinger-transmission-eigenvalues} and \ref{helmholtz-transmission-eigenvalues}. We note that the behaviour of transmission eigenvalues for the Schr\"odinger operator $H_0 + V - \lambda$ is different than for the Helmholtz operator $H_0 + \lambda V - \lambda$. The number $\nu\in\left\{0,1\right\}$ shall denote our choice of operator, and let us assume that $\Omega\subseteq\mathbb H^n$ is a bounded nonempty open set, and that $V\in L^\infty(\Omega)$ satisfies the conditions of one of Theorems \ref{schrodinger-transmission-eigenvalues} and \ref{helmholtz-transmission-eigenvalues}. The first step in the proofs is to move from the interior transmission problem to a single fourth-order equation.
\begin{proposition}
Let $\nu\in\left\{0,1\right\}$, and let $\lambda\in\mathbb R$ or $\lambda\in\mathbb R\setminus\{0\}$ depending on whether $\nu=0$ or $\nu=1$, let $\Omega$ be a bounded nonempty open set in $\mathbb H^n$, and let $V\in L^\infty(\Omega)$ only take positive real values, or only negative real values, and be bounded away from zero. Then the system
\[\begin{cases}
(H_0+\lambda^\nu V-\lambda)v=0&\text{in $\Omega$,}\\
(H_0-\lambda)w=0&\text{in $\Omega$,}\\
v-w\in H_0^2(\Omega)
\end{cases}\]
has a solution $v,w\in L^2(\Omega)$ with $v\not\equiv0$ and $w\not\equiv0$ \textbf{if and only if} there exists a function $u\in H_0^2(\Omega)$ with $u\not\equiv0$ solving (in the sense of distributions) the fourth-order equation
\[\left(H_0+\lambda^\nu V-\lambda\right)\frac1V\left(H_0-\lambda\right)u=0\]
in $\Omega$. Furthermore, this transition retains multiplicities in the sense that the space of pairs of solutions $\left\langle v,w\right\rangle$ to the interior transmission problem has the same dimension as the space of solutions $u$ to the fourth-order equation.
\end{proposition}

\begin{proof}
If $\lambda$ is such that the interior transmission problem has a non-trivial solution $v,w\in L^2(\Omega)$, then it is a simple matter of computation to show that $u=v-w\in H_0^2(\Omega)$ solves the fourth-order equation. Furthermore, $u\not\equiv0$ because otherwise we would have $Vv\equiv0$, which is not possible.

Conversely, if the fourth-order equation has a non-trivial solution $u\in H_0^2(\Omega)$, then it is a routine calculation to check that the functions
\[v=-\frac1{\lambda^\nu\,V}\left(H_0-\lambda\right)u\quad\text{and}\quad w=-\frac1{\lambda^\nu\,V}\left(H_0+\lambda^\nu V-\lambda\right)u,\]
where $\lambda^\nu$ is understood as $1$ if $\nu=0$ and $\lambda=0$,
are in $L^2(\Omega)$ and solve the interior transmission problem. Furthermore, we have $v-w=u\in H_0^2(\Omega)$. Neither of $v$ and $w$ can vanish identically because $v-w \in H_0^2(\Omega)$ would then imply that $v \equiv w \equiv u \equiv 0$ by the unique continuation principle.

Finally, the above linear mappings between the spaces of solutions from $\left\langle v,w\right\rangle$ to $u=v-w$, and from $u$ to $\left\langle v,w\right\rangle$, are injections and inverses of each other, which yields the last statement of the proposition.
\end{proof}

We now move into the realm of quadratic forms and their analytic perturbation theory. First, we define the quadratic form $Q_\lambda\colon H^2_0(\Omega)\longrightarrow\mathbb C$, for every $\lambda\in\mathbb C$, by setting
\[Q_\lambda(u)=\left\langle\left(H_0+\overline{\lambda^\nu}V-\overline\lambda\right)u\middle|\frac1{\left|V\right|}\left(H_0-\lambda\right)u\right\rangle_{L^2(\Omega;\mathrm d\mu)}.\]
Here the operator on the left is the adjoint operator of $H_0+\lambda^\nu V-\lambda$ in the hyperbolic inner product $\langle f | g \rangle_{L^2(\Omega;\mathrm d\mu)} = \int_\Omega \overline{f}\,g\,\mathrm d\mu$ for $f, g \in L^2(\Omega;\mathrm d\mu)$.

In particular, for $\lambda\in\mathbb R$, we have
\[Q_\lambda(u)=\left\|\frac1{\sqrt{\left|V\right|}}\left(H_0-\lambda\right)u\right\|_{L^2(\Omega,\mathrm d\mu)}^2+\lambda^\nu\frac{V}{\left|V\right|}\left\langle u\middle|\left(H_0-\lambda\right)u\right\rangle_{L^2(\Omega;\mathrm d\mu)}\]
since $V/\left|V\right|$ is a constant as $V$ is real-valued and has a constant sign.
Following the presentations of \cite{Vesalainen1, Vesalainen2}, where Euclidean transmission eigenvalues were considered for suitable polynomially and exponentially decaying potentials in unbounded domains, almost verbatim, we may list the relevant facts about $Q_\lambda$ which follow from analytic perturbation theory \cite{Kato}. The types (a) and (B) are defined there in Chapter VII.
\begin{proposition}\label{quadratic-forms}
The quadratic forms $Q_\lambda$ form an entire self-adjoint analytic family of quadratic forms of type (a) with compact resolvent, and so give rise to an entire self-adjoint analytic family of operators $T_\lambda$ of type (B) with compact resolvent.

There exists a sequence of real-analytic functions $\mu_\ell\colon\mathbb R\longrightarrow\mathbb R$, $\ell\in\mathbb Z_+$, such that for each $\lambda\in\mathbb R$, the spectrum of $T_\lambda$, which consists of a discrete set of real eigenvalues of finite multiplicities and accumulating to $+\infty$, is given by the values $\mu_\ell(\lambda)$, $\ell\in\mathbb Z_+$, respecting multiplicities.

Finally, for any fixed $T\in\mathbb R_+$, there exist constants $c,C\in\mathbb R_+$ such that
\[\left\lvert\mu_\ell(\lambda)-\mu_\ell(0)\right\rvert\leqslant C\left(e^{c\left|\lambda\right|}-1\right),\]
uniformly for $\ell\in\mathbb Z_+$ and $\lambda\in\left[-T,T\right]$.
\end{proposition}

\begin{figure}
\begin{center}
\includegraphics{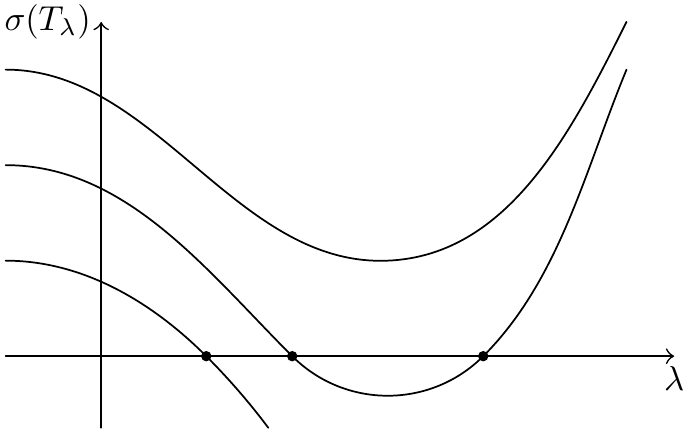}
\end{center}
\caption{The situation described in Propositions \ref{quadratic-forms} and \ref{connection-to-quadratic-forms}. There exists a family of real-analytic functions $\mu_\ell(\lambda)$ which give the eigenvalues of $T_\lambda$, respecting multiplicities. Transmission eigenvalues correspond to zeros of the functions $\mu_\ell(\lambda)$.}
\end{figure}
The connection to the fourth-order equation, and a fortiori to the interior transmission problem is established easily with the min-max principle:
\begin{proposition}\label{connection-to-quadratic-forms}
For each $\lambda\in\mathbb R$, the dimension of the space of solutions $u\in H^2_0(\Omega)$ to
\[\left(H_0+\lambda^\nu\,V-\lambda\right)\frac1V\left(H_0-\lambda\right)u=0\]
in $\Omega$ is equal to the number the pairs $\left\langle\ell,\lambda\right\rangle\in\mathbb Z_+\times\mathbb R$ for which $\mu_\ell(\lambda)=0$.
\end{proposition}
In particular, $\lambda\in\mathbb R$ is a transmission eigenvalue if and only if $0$ is an eigenvalue of the fourth-order operator $T_\lambda$. Again the proof is the same as in~\cite{Vesalainen1}.

Our next goal is to prove the existence of Helmholtz transmission eigenvalues for constant potentials in discs of $\mathbb H^n$. This is a key step in establishing the existence of transmission eigenvalues for more general potentials via fourth-order operators and their quadratic forms in both of the cases $\nu=0$ and $\nu=1$. 

\begin{proposition}\label{constant-potentials}
Let $R\in\mathbb R_+$, and let $V_0\in\left]-\infty,1\right[$ be a constant potential
in an open ball $B \subset \mathbb H^n$ of radius~$R$. Then there exists an infinite
sequence of positive real numbers $\lambda$, tending to $+\infty$, such that
the system
\[\begin{cases}
\left(H_0+\lambda V_0-\lambda\right)v=0&\text{in $B$,}\\
\left(H_0-\lambda\right)w=0&\text{in $B$,}\\
v-w\in H^2_0(B)
\end{cases}\]
has a solution $v,w\in L^2(B)$ with $v\not\equiv0$ and $w\not\equiv0$. Furthermore, these $\lambda$ are Helmholtz non-scattering energies for the potential $V_0\,\chi_B$ in $\mathbb H^n$, and each of them is a Schr\"odinger non-scattering energy for the corresponding scaled potential $\lambda\,V_0\,\chi_B$.
\end{proposition}


\begin{proof}
We shall follow the approach of the proof of Theorem 2 in
\cite{Colton--Paivarinta--Sylvester} and look for radial solutions to the
system. Instead of the hyperbolic polar coordinates $\left\langle r,\vartheta\right\rangle$ originating 
from the center of $B$, we shall use the 
coordinates $\left\langle\rho,\vartheta\right\rangle$ with $\rho$ related to $r$ by $\cosh r=2\rho+1$ and where
$\vartheta\in S^{n-1}$ represents the angular variables. We can also translate $B$ so that its center is the point $\langle 0,\ldots,0,1 \rangle$ in the $\langle x', x_n \rangle$ coordinates of the upper half-space model of $\mathbb H^n$.

We need a formula for $H_0$ applied to a radially symmetric function. If $x, y \in \mathbb H^n$ then the hyperbolic distance between $x$ and $y$ has the pleasant expression $ \operatorname{arcosh}(1+\abs{x-y}^2/(2x_ny_n))$. Hence the radial distance of $x$ to the center of $B$ is given by
\[
    r = \operatorname{arcosh}\left( 1 + \frac{\abs{x'}^2+\abs{x_n-1}^2}{2x_n} \right),
\]
and hence
\[
    \rho = \frac{\abs{x'}^2+x_n^2-2x_n+1}{4x_n}.
\]
Also, we see that for $j\in\left\{1,2,\ldots,n-1\right\}$,
\begin{align*}
    &\frac{\partial\rho}{\partial x_j} = \frac{x_j}{2x_n},
    && \frac{\partial\rho}{\partial x_n} = \frac{x_n^2-\abs{x'}^2-1,}{4x_n^2},\\
&\frac{\partial^2 \rho}{\partial x_j^2} = \frac{1}{2x_n},
&& \frac{\partial^2 \rho}{\partial x_n^2} = \frac{\abs{x'}^2+1}{2x_n^3}.
\end{align*}
Let us consider a sufficiently smooth function $f$ defined in a non-empty open subset of $\mathbb H^n$, and let us assume that $f$ is radially symmetric, which means that it is constant on the sphere of the form $\left\{\rho=c\right\}$ for any fixed $c\in\mathbb R_+$. Hence
\[\frac{\partial f}{\partial x_j} = \frac{\partial f}{\partial\rho} \cdot\frac{\partial \rho}{\partial x_j}\qquad
\text{and}\qquad
\frac{\partial^2 f}{\partial x_j^2} = \frac{\partial^2 f}{\partial\rho^2}\cdot \left(\frac{\partial\rho}{\partial {x_j}}\right)^{\!2} + \frac{\partial f}{\partial\rho} \cdot\frac{\partial^2 \rho}{\partial x_j^2}.\]

The change of coordinates of the previous paragraph applied to \eqref{H0def} gives a formula which simplifies to
\[
    H_0=-\rho\left(\rho+1\right)\frac{\partial^2}{\partial \rho^2}-\left(n\rho+\frac
    n2\right)\frac\partial{\partial\rho}-\frac{(n-1)^2}4.
\]
In the above coordinates $B$ has radius $\mathrm P=(\cosh R-1)/2$. A
radial solution $v=v(\rho)$, $w=w(\rho)$ to our system must now solve the ordinary
differential equations
\[
-\rho\left(\rho+1\right)\frac{\partial^2v}{\partial\rho^2}-\left(n\rho+\frac
n2\right)\frac{\partial v}{\partial\rho}-\frac{(n-1)^2v}4+\lambda V_0v-\lambda
v=0,
\]
and
\[
-\rho\left(\rho+1\right)\frac{\partial^2w}{\partial\rho^2}-\left(n\rho+\frac
n2\right)\frac{\partial w}{\partial\rho}-\frac{(n-1)^2w}4-\lambda w=0.
\]
In terms of the variable $z=-\rho$, this simplifies into the hypergeometric equation
\[
    z(1-z)\frac{\p^2 f}{\p z^2} + (c-(a+b+1)z)\frac{\p f}{\p z} -ab f = 0,
\]
with certain constants $a$, $b$ and $c$ (see Section 15.10 in \cite{NIST}). 
The solutions to this equation can be given by means of the hypergeometric functions
$F(a,b,c;z)=\vphantom F_2F_1\!(a,b,c;z)$. In particular, we have the solutions
\[
    v(\rho)=A\cdot F\!\left(\frac{n-1}2-i\sqrt{\lambda-\lambda
    V_0},\frac{n-1}2+i\sqrt{\lambda-\lambda V_0},\frac n2;-\rho\right),
\]
and
\[
w(\rho)=B\cdot F\!\left(\frac{n-1}2-i\sqrt{\lambda},\frac{n-1}2+i\sqrt{\lambda},\frac
n2;-\rho\right).
\]
for some real constant coefficients $A$ and $B$. These functions are real-valued. Furthermore, since $v$ and $w$ are bounded in the neighbourhood of the center $\gamma$ of the ball $B$, and since the original equations have smooth coefficients, $v$ and $w$ solve their respective equations also in the neighbourhood of $\gamma$ (see e.g.\ \cite{Serrin}).

We would like to pick nonzero constants $A$ and $B$ so that the required boundary conditions
\[v(\mathrm P)=w(\mathrm P)\quad\text{and}\quad v'(\mathrm P)=w'(\mathrm P)\]
hold. The derivatives $v'(\mathrm P)$ and $w'(\mathrm P)$ are given by
\begin{multline*}
v'(\mathrm P)=-A\cdot\frac{\left((n-1)/2\right)^2+\lambda-\lambda
V_0}{n/2}\\\cdot F\!\left(\frac{n+1}2-i\sqrt{\lambda-\lambda
V_0},\frac{n+1}2+i\sqrt{\lambda-\lambda V_0},\frac{n+2}2;-\mathrm P\right),
\end{multline*}
and
\begin{multline*}
w'(\mathrm P)=-B\cdot\frac{\left((n-1)/2\right)^2+\lambda}{n/2}\cdot F\!\left(\frac{n+1}2-i\sqrt{\lambda},\frac{n+1}2+i\sqrt{\lambda},\frac{n+2}2;-\mathrm P\right),
\end{multline*}
(see e.g. Section 15.5 in \cite{NIST}).
Thus, $A$ and $B$ need to solve the homogeneous pair of linear equations
\[\begin{cases}
A\cdot F(\ldots)-B\cdot F(\ldots)=0,\\
A\,\bigl(\left((n-1)/2\right)^2+\lambda-\lambda V_0\bigr)\, F(\ldots)
-B\,\bigl(\left((n-1)/2\right)^2+\lambda\bigr)\, F(\ldots)=0.
\end{cases}\]
We will not choose any values of $\lambda$, $A$ or $B$ explicitly. Rather, we
shall show that as $\lambda\longrightarrow\infty$, the determinant of the above
system will have infinitely many zeros. The implicit constants in the $O$-terms below may depend on $n$, $R$ and $V_0$, but not on $\lambda$.

Using the asymptotics from sections 15.12(iii) and 10.17(i) of \cite{NIST}, we obtain rather pleasant asymptotics in terms of $\lambda\longrightarrow\infty$, for we have
\begin{align*}
&F\!\left(\frac{n-1}2-i\sqrt\lambda,\frac{n-1}2+i\sqrt\lambda,\frac
n2;-\mathrm P\right)\\
&\qquad=C_{n,R}\,\lambda^{1/4-n/4}\,\cos\!\left(R\sqrt\lambda-\frac{\left(n-1\right)\pi}4\right)+O(\lambda^{-1/4-n/4}),
\end{align*}
as $\lambda\longrightarrow\infty$, where $C_{n,R}$ is a nonzero
real coefficient only depending on $n$ and $R$. More precisely, we have used here the asymptotics (15.12.5) from \cite{NIST}, rewritten the $I$-Bessel function in terms of the $J$-Bessel function using (10.27.6) of \cite{NIST}, and then applied the asymptotics of the $J$-Bessel function given by (10.17.3) of \cite{NIST}.
Similarly, we may derive asymptotics for the
expressions involving the derivatives:
\begin{align*}
&\left(\frac{(n-1)^2}4+\lambda\right)F\!\left(\frac{n+1}2-i\sqrt\lambda,\frac{n+1}2+i\sqrt\lambda,\frac
{n+2}2;-\mathrm P\right)\\
&\quad=\widetilde
C_{n,R}\,\lambda^{3/4-n/4}\,\cos\!\left(R\sqrt\lambda-\frac{\left(n+1\right)\pi}4\right)+O(\lambda^{1/4-n/4}),
\end{align*}
as $\lambda\longrightarrow\infty$, where $\widetilde C_{n,R}$ is again a nonzero real coefficient only depending on $n$ and $R$.

Combining the above facts, we see, after some simplifications, that the asymptotics for the determinant of the pair of equations is
\begin{align*}
\det(\lambda)
&=D_{n,R,V_0}\,\lambda^{1-n/2}\,M(\lambda)+O(\lambda^{1/2-n/2}),
\end{align*}
as $\lambda\longrightarrow\infty$, where
\begin{align*}
M(\lambda)&=\left(1-\sqrt{1-V_0}\right)\cos\!\left(R(\sqrt\lambda+\sqrt{\lambda-\lambda V_0})-\frac{n\pi}2\right)\\
&\qquad+\left(1+\sqrt{1-V_0}\right)\sin\!\left(R(\sqrt\lambda-\sqrt{\lambda-\lambda V_0})\right),
\end{align*}
where $D_{n,R,V_0}$ is a nonzero real constant only depending on $n$, $R$ and $V_0$. Now the existence claim follows from the fact that the second term clearly dominates the first, when its absolute value reaches local maxima, thereby giving rise to infinitely many sign changes as $\lambda\longrightarrow\infty$.

It only remains to check that these values $\lambda$ obtained are in fact non-scattering energies. This is done by observing that $w$ clearly defines a real-analytic function in all of $\mathbb H^n$, and $v$ can be extended as $w$ to a twice weakly differentiable function in all of $\mathbb H^n$. It is clear that as $v-w$ is compactly supported, it is in $\mathring B^\ast(\mathbb H^n)$. Thus, we only need to show that $v,w\in\mathbb B^\ast(\mathbb H^n)$. However, since we are only concerned with the asymptotic behaviour far away, it is enough to check that $w\in\mathbb B^\ast(\mathbb H^n)$.

Without loss of generality, we may assume that the center of the ball $B$ is the point $O=\left\langle0,\ldots,0,1\right\rangle$ in the usual $\left\langle x',x_n\right\rangle$ coordinates of the upper half-space model of $\mathbb H^n$. Then the properties of hypergeometric functions, namely Paragraph \S15.2(i) and Equation (15.8.2) of Paragraph \S15.8(i) in \cite{NIST}, give the estimate
\[\left|w(\rho)\right|^2\leqslant C\,\rho^{1-n},\]
as $\rho\longrightarrow\infty$,
which is $\leqslant C\,e^{(1-n)r}$, where $r$ is the hyperbolic distance from the point $O$. Since
\[\cosh r=\frac{\left|x'\right|^2+x_n^2+1}{2x_n},\]
as $\rho\longrightarrow\infty$ (see e.g.\ Section\ 1.1 in \cite{Isozaki--Kurylev}), we may estimate
\[\left|w(x)\right|^2\leqslant C\left(\left|x'\right|^2+x_n^2+1\right)^{1-n}x_n^{n-1}.\]

Now, let $R\in\left]e,\infty\right[$. The expression in the $B^\ast$-norm of $w$ is now
\[\frac1{\log R}\int\limits_{1/R}^R\int\limits_{\mathbb R^{n-1}}\left|w\right|^2\mathrm dx'\,\frac{\mathrm dx_n}{x_n^n}
\leqslant\frac C{\log R}\int\limits_{1/R}^R\int\limits_{\mathbb R^{n-1}}\left(\left|x'\right|^2+x_n^2+1\right)^{1-n}\mathrm dx'\,\frac{\mathrm dx_n}{x_n}.\]
We split the $x'$-integral into two parts according to whether $\left|x'\right|\geqslant\sqrt{x_n^2+1}$ or not, and we may continue
\begin{align*}
&\leqslant\frac C{\log R}\int\limits_{1/R}^R\left(\,\int\limits_{\left|x'\right|\geqslant\sqrt{x_n^2+1}}\left|x'\right|^{2-2n}\mathrm dx'
+\int\limits_{\left|x'\right|\leqslant\sqrt{x_n^2+1}}\left(x_n^2+1\right)^{1-n}\mathrm dx'\right)\frac{\mathrm dx_n}{x_n}\\
&\leqslant\frac C{\log R}\int\limits_{1/R}^R\left(x_n^2+1\right)^{(1-n)/2}\,\frac{\mathrm dx_n}{x_n}
\leqslant\frac C{\log R}\int\limits_{1/R}^1\frac{\mathrm dx_n}{x_n}
+\frac C{\log R}\int\limits_1^Rx_n^{-n}\,\mathrm dx_n\leqslant C,
\end{align*}
where $C$ only depends on $w$ and $n$, but not on $R$, and so we have proved that $\left\|w\right\|_{B^\ast(\mathbb H^n)}<\infty$, as required.
\end{proof}

\begin{proof}[Proof of Theorem~\ref{schrodinger-transmission-eigenvalues}.]
It is easy to check that for real potentials $V$, the transmission eigenvalues are indeed real, as the transmission eigenvalue system implies $\Im\lambda\,\left\|v\right\|_{L^2(\Omega;\mathrm d\mu)}^2=\Im\lambda\,\left\|w\right\|_{L^2(\Omega;\mathrm d\mu)}^2=0$. Recall that $\nu=0$ in the Schr\"odinger case and hence we will consider the quadratic form
\[Q_\lambda(u)=\left\|\frac1{\sqrt{V}}\left(H_0-\lambda\right)u\right\|_{L^2(\Omega;\mathrm d\mu)}^2+\left\langle u\middle|\left(H_0-\lambda\right)u\right\rangle_{L^2(\Omega;\mathrm d\mu)},\]
for $u\in H^2_0(\Omega)$ and $\lambda\in\mathbb R$.
It is also easy to observe that $Q_\lambda(u)\geqslant-\lambda\left\|u\right\|_{L^2(\Omega;\mathrm d\mu)}^2$ for all $u\in H^2_0(\Omega)$ and $\lambda\in\mathbb R$. Thus, for $\lambda\in\mathbb R_-$, the eigenvalues of $T_\lambda$ from Proposition~\ref{quadratic-forms} are all positive, and the eigenvalues of $T_0$ are nonnegative. Furthermore, if $0$ was an eigenvalue of $T_0$, then we would have $Q_0(u)=0$ for a corresponding eigenvector $u\in\mathrm{Dom}\,T_0\subseteq H_0^2(\Omega)$. But then $H_0u=0$ in $\Omega$ since both terms of $Q_0(u)$ have to vanish by their non-negativity. Hence by the $H_0^2(\Omega)$-condition the zero extension $\widetilde u$ of $u$ into all of $\mathbb H^n$ is a solution to $H_0\widetilde u=0$ in $\mathbb H^n$, and by real-analyticity $\widetilde u\equiv0$ in all of $\mathbb H^n$, which gives a contradiction. Thus, all eigenvalues of $T_0$ are positive as well.

\begin{figure}
\begin{center}
\includegraphics{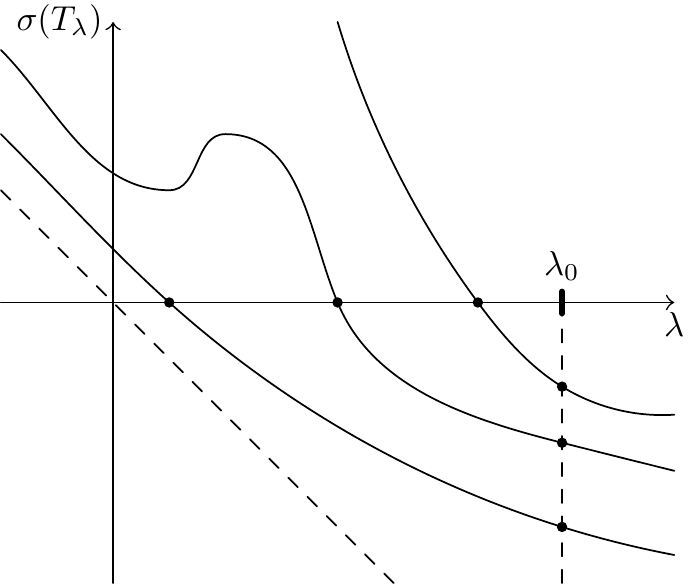}
\end{center}
\caption{The situation in the proof of Theorem 8. The spectrum of $T_\lambda$ is contained in $\left[-\lambda,\infty\right[$, and in particular all the values $\mu_\ell(\lambda)$ are strictly positive for $\lambda\in\left]-\infty,0\right]$. The key idea in the proof of existence is to find a number $\lambda_0\in\mathbb R_+$ such that there is a subspace $X$ of dimension $N$ such that $Q_\lambda$ is negative-semidefinite in $X$. This implies that at least $N$ of the values $\mu_\ell(\lambda_0)$ are non-positive and so the corresponding functions must intersect the $\lambda$-axis at least $N$ times, and so we obtain at least $N$ transmission eigenvalues.}
\end{figure}

The discreteness follows directly from Propositions \ref{quadratic-forms} and \ref{connection-to-quadratic-forms}: for any $T\in\mathbb R_+$ only finitely many of the eigenvalues $\mu_\ell(\lambda)$, $\ell\in\mathbb Z_+$, can vanish for some $\lambda\in\left[-T,T\right]$. This is because they are positive at $\lambda=0$, tend to $+\infty$ as $\ell\longrightarrow\infty$ and because $\lvert\mu_\ell(\lambda)-\mu_\ell(0)\rvert\ll_T e^{c\lvert\lambda\rvert}-1$. Furthermore, because of real-analyticity, for each $\ell$, there are at most finitely many zeros in $\left[-T,T\right]$.

To prove the existence statement let us emphasize that in any ball and despite being written for the Helmholtz operator, Proposition~\ref{constant-potentials} guarantees the existence of pairs $\left\langle\lambda_0, V_0\right\rangle \in \mathbb R_+ \times \mathbb R_+$ such that $\lambda_0$ is a transmission eigenvalue for the Schr\"odinger operator $H_0 + V_0 - \lambda$ in that ball. Suppose that $\Omega$ contains congruent pairwise disjoint balls $B_1$, $B_2$, \dots, $B_N$ (with $N\in\mathbb Z_+$). Let $\lambda_0\in\mathbb R_+$ be a transmission eigenvalue for the constant potential $V_0$ in each of the balls. Let us define an auxiliary quadratic form
\[\widetilde Q_\lambda(u)=\left\|\frac1{\sqrt{V_0}}\left(H_0-\lambda\right)u\right\|^2_{L^2(\Omega;\mathrm d\mu)}+\left\langle u\middle|\left(H_0-\lambda\right)u\right\rangle_{L^2(\Omega;\mathrm d\mu)}\]
for $\lambda\in\mathbb R$ and $u\in H_0^2(\Omega)$.

Since $\lambda_0$ is a transmission eigenvalue for the constant potential $V_0$ in the balls, we have functions $u_1,\ldots,u_N\in H_0^2(\Omega)$, supported in the balls $\overline B_1$, \dots, $\overline{B_N}$, respectively, such that
$\widetilde Q_{\lambda_0}(u)=0$ for $u\in X=\mathrm{span}\left\{u_1,\ldots,u_N\right\}$. Now, if $V\geqslant V_0$ in each of the balls, it is easy to check that
\[Q_{\lambda_0}(u)\leqslant\widetilde Q_{\lambda_0}(u)=0\]
for $u$ in the $N$-dimensional space $X$. This means that at least $N$ of the eigenvalues $\mu_\ell(\lambda)$, $\ell\in\mathbb Z_+$, must vanish for some $\lambda\in\left]0,\lambda_0\right]$, thereby establishing the existence of at least $N$ transmission eigenvalues, counting multiplicities.
\end{proof}

\begin{proof}[Proof of Theorem~\ref{helmholtz-transmission-eigenvalues}.]
In the Helmholtz case the quadratic form becomes
\[Q_\lambda(u)=\left\|\frac1{\sqrt{\left|V\right|}}\left(H_0-\lambda\right)u\right\|_{L^2(\Omega,\mathrm d\mu)}^2+\lambda\frac{V}{\left|V\right|}\left\langle u\middle|\left(H_0-\lambda\right)u\right\rangle_{L^2(\Omega;\mathrm d\mu)}\]
for $u\in H^2_0(\Omega)$ and $\lambda\in\mathbb R$,
where $V/\left|V\right|$ is the constant $+1$ or $-1$ depending on the fixed sign of $V$. 
Discreteness is obtained in essentially the same way as in the proof of Theorem~\ref{schrodinger-transmission-eigenvalues}: we observe that $Q_0(u)\geqslant0$ for all $u\in H_0^2(\Omega)$, and if $0$ was an eigenvalue of $T_0$, then an eigenfunction $u\in\mathrm{Dom}\,T_0\subseteq H_0^2(\Omega)$ would solve $H_0u=0$ in $\Omega$, leading to a contradiction as before. Thus all the eigenvalues of $T_0$ are strictly positive positive. Now the result follows from Propositions \ref{quadratic-forms} and \ref{connection-to-quadratic-forms}, since for any $T\in\mathbb R_+$ at most finitely many eigenvalues $\mu_\ell(\lambda)$, $\ell\in\mathbb Z_+$, can vanish for some $\lambda\in\left[-T,T\right]$, and each $\mu_\ell$ has at most finitely many zeros.

The existence claim is also established in the same manner as in the proof of Theorem~\ref{schrodinger-transmission-eigenvalues}. We pick an arbitrarily large $N\in\mathbb Z_+$ and a radius $r\in\mathbb R_+$ so small that the domain $\Omega$ contains $N$ pairwise disjoint balls $B_1$, \dots, $B_N$ of radius $r$. Let $V_0\in\left]-\infty,1\right[\setminus\left\{0\right\}$ be such that $\left|V_0\right|\leqslant\left|V\right|$ and $V_0/\left|V_0\right|=V/\left|V\right|$ in $B_1\cup\ldots\cup B_N$. We define the auxiliary quadratic form
\[\widetilde Q_\lambda(u)=\left\|\frac1{\sqrt{\left|V_0\right|}}\left(H_0-\lambda\right)u\right\|^2_{L^2(\Omega;\mathrm d\mu)}+\lambda\,\frac{V_0}{\left|V_0\right|}\left\langle u\middle|\left(H_0-\lambda\right)u\right\rangle_{L^2(\Omega;\mathrm d\mu)}\]
for $u\in H_0^2(\Omega)$ and $\lambda\in\mathbb R$.

Now, by Proposition~\ref{constant-potentials} there exists $\lambda_0\in\mathbb R_+$ that is a transmission eigenvalue for the constant potential $V_0$ in each of the balls. Again we get an $N$-dimensional subspace $X\subset H_0^2(\Omega)$ such that $\widetilde Q_{\lambda_0}(u)=0$ for $u\in X$, and clearly $Q_{\lambda_0}(u)\leqslant\widetilde Q_{\lambda_0}(u)=0$ for $u\in X$, therefore establishing the existence of $N$ transmission eigenvalues in $\left]0,\lambda_0\right]$, counting multiplicities.

Finally, for negatively valued $V$, the simple observation that $Q_\lambda(u)>0$ for $u\in H_0^2(\Omega)\setminus\left\{0\right\}$ and $\lambda\in\mathbb R_-$ shows immediately that there are no negative transmission eigenvalues.
\end{proof}

\section*{Acknowledgements}

The authors would like to express their gratitude to Valter Pohjola for many useful discussions during the research of this paper. The second author would also like to thank Jouni Parkkonen for conversations regarding analysis in hyperbolic spaces.

The first author was funded in part by the European Research Council's 2010 Advanced Grant 267700.

The second author was funded by the Academy of Finland through the Finnish Centre of Excellence in Inverse Problems Research and the projects 276031, 282938, 283262 and 303820, by the Magnus Ehrnrooth Foundation, and by the Basque Government through the BERC 2014--2017 program and by Spanish Ministry of Economy and Competitiveness MI\-NE\-CO: BCAM Severo Ochoa excellence accreditation SEV-2013-0323.

The authors would like to express their gratitude for the generous financial support of the Henri Poincar\'e Institute, where part of this research was carried out during the thematic program of inverse problems in 2015.


\footnotesize
\begin{thebibliography}{99}\setlength{\itemsep}{-1pt}

\bibitem{Baouendi--Zachmanoglou}
\textsc{Baouendi, M. S.}, and \textsc{E. C. Zachmanoglou}:
\textit{Unique continuation theorems for solutions of partial differential equations}, Bull. Amer. Math. Soc., 83 (1977), 1045--1048.

\bibitem{Bers--John--Schechter}
\textsc{Bers, L.}, \textsc{F. John}, and \textsc{M. Schechter}: \textit{Partial Differential Equations}, Lectures in Applied Mathematics III, John Wiley \& Sons, 1964.

\bibitem{Blasten--Li--Liu--Wang}
\textsc{Bl{\aa}sten, E.}, \textsc{X. Li}, \textsc{H. Liu} and \textsc{Y. Wang}:
\textit{On vanishing and localizing of transmission eigenfunctions near singular points: a numerical study}, Inverse Problems, 33 (2017), 105001.

\bibitem{Blasten--Liu2017}
\textsc{Bl{\aa}sten, E.}, and \textsc{H. Liu}:
\textit{On vanishing near corners of transmission eigenfunctions},
Journal of Functional Analysis, 273 (2017), 3616--3632.

\bibitem{Blasten--Liu2017b}
\textsc{Bl{\aa}sten, E.}, and \textsc{H. Liu}:
\textit{Recovering piecewise constant refractive indices by a single far-field pattern}, arXiv:1705.00815.

\bibitem{Blasten--Paivarinta--Sylvester}
\textsc{Bl{\aa}sten, E.}, \textsc{L. P\"aiv\"arinta}, and \textsc{J. Sylvester}:
\textit{Corners always scatter},
Comm.\ Math.\ Phys., 331 (2014), 725--753.

\bibitem{Cakoni--Colton--Monk}
\textsc{Cakoni, F.}, \textsc{D. Colton}, and \textsc{P. Monk}: \textit{On the use of transmission eigenvalues to estimate the index of refraction from far field data}, Inverse Problems, 23 (2007), 507--522.

\bibitem{Cakoni--Gintides--Haddar}
\textsc{Cakoni, F.}, \textsc{D. Gintides}, and \textsc{H. Haddar}: \textit{The existence of an infinite discrete set of transmission eigenvalues}, SIAM J. Math. Anal., 42 (2010), 237--255.

\bibitem{Cakoni--Haddar1}
\textsc{Cakoni, F.}, and \textsc{H. Haddar}: \textit{Transmission eigenvalues in inverse scattering theory}, in \cite{Uhlmann}, 529--578.

\bibitem{Cakoni--Haddar2}
\textsc{Cakoni, F.}, and \textsc{H. Haddar}: \textit{Transmission eigenvalues}, Inverse Problems, 29 (2013), 100201, 1--3.

\bibitem{Colton--Kirsch}
\textsc{Colton, D.}, and \textsc{A. Kirsch}: \textit{A simple method for solving inverse scattering problems in the resonance region}, Inverse Problems, 12 (1996), 383--393.

\bibitem{Colton--Kirsch--Paivarinta}
\textsc{Colton, D.}, \textsc{A. Kirsch}, and \textsc{L. P\"aiv\"arinta}: \textit{Far field patterns for acoustic waves in an inhomogeneous medium}, SIAM J. Math. Anal., 20 (1989), 1472--1483.

\bibitem{Colton--Kress}
\textsc{Colton, D.}, and \textsc{R. Kress}: \textit{Inverse Acoustic and Electromagnetic Scattering Theory}, Applied Mathematical Sciences, 93, Springer, 2013.

\bibitem{Colton--Monk}
\textsc{Colton, D.}, and \textsc{P. Monk}: \textit{The inverse scattering problem for acoustic waves in an inhomogeneous medium}, Quart. J. Mech. Appl. Math., 41 (1988), 97--125.

\bibitem{Colton--Paivarinta--Sylvester}
\textsc{Colton, D.}, \textsc{L.\ P\"aiv\"arinta}, and \textsc{J.\ Sylvester}: \textit{The interior transmission problem}, Inverse Probl.\ Imaging, 1 (2007), 13--28.

\bibitem{Cordes}
\textsc{Cordes, H.}: \textit{\"Uber die eindeutige Bestimmtheit der L\"osungen elliptischer Differentialgleichunger durch Anfangsvorgaben}, Nachr. Akad. Wiss. G\"ottingen, IIa (1956), 239--258.

\bibitem{Elschner--Hu}
\textsc{Elschner, J.}, and \textsc{G.\ Hu}: \textit{Corners and edges always scatter}, Inverse Problems, 31 (2015), 015003, 1--17.

\bibitem{Elschner--Hu2}
\textsc{Elschner, J.}, and \textsc{G.\ Hu}: \textit{Acoustic Scattering from Corners, Edges and Circular Cones}, Arch. Ration. Mech. Anal., 228 (2018), 653--690.

\bibitem{HormanderII}
\textsc{H\"ormander, L.}: \textit{The Analysis of Linear Partial Differential Operators II: Differential Operators with Constant Coefficients}, Classics in Mathematics, Springer, 2005.

\bibitem{Hu--Salo--Vesalainen}\textsc{Hu, G.}, \textsc{M.\ Salo}, and \textsc{E.\ V.\ Vesalainen}: \textit{Shape identification in inverse medium scattering with a single far-field pattern}, SIAM J. Math. Anal., 48 (2016), 152--165.

\bibitem{Isozaki--Kurylev}
\textsc{Isozaki, H.}, and \textsc{Y.\ Kurylev}: \textit{Introduction to Spectral Theory and Inverse Problem on Asymptotically Hyperbolic Manifolds}, Mathematical Society of Japan Memoirs, 32, Mathematical Society of Japan, 2014.

\bibitem{Kato}
\textsc{Kato, T.}: \textit{Perturbation Theory for Linear Operators}, Classics in Mathematics, Springer, 1995.

\bibitem{Kenig--Ruiz--Sogge}
\textsc{Kenig, C.E.}, \textsc{A. Ruiz}, and \textsc{C. D. Sogge}: \textit{Uniform Sobolev inequalities and unique continuation for second order constant coefficient differential operators}, Duke Math. J., 55 (1987), 329--347.

\bibitem{Kirsch}
\textsc{Kirsch, A.}: \textit{The denseness of the far field patterns for the transmission problem}, IMA J. Appl. Math., 37 (1986), 213--225.

\bibitem{Kirsch1}
\textsc{Kirsch, A.}: \textit{Characterization of the shape of a scattering obstacle using the spectral data of the far field operator}, Inverse Problems, 14 (1998), 1489--1512.

\bibitem{Kirsch2}
\textsc{Kirsch, A.}: \textit{Factorization of the far field operator for the inhomogeneous medium case and an application in inverse scattering theory}, Inverse Problems, 15 (1999), 413--429.

\bibitem{Xiao--Liu}
\textsc{Liu, H.}, and \textsc{J. Xiao}:
\textit{On electromagnetic scattering from a penetrable corner}, SIAM J. Math. Anal., 49 (2017), 5207--5241.

\bibitem{McLaughlin--Polyakov}
\textsc{McLaughlin, J. R.}, and \textsc{P. L. Polyakov}: \textit{On the uniqueness of a spherically symmetric speed of sound from transmission eigenvalues}, J. Differential Equations, 107 (1994), 351--382.

\bibitem{McLaughlin--Polyakov--Sacks}
\textsc{McLaughlin, J. R.}, \textsc{P. L. Polyakov}, and \textsc{P. E. Sacks}: \textit{Reconstruction of a spherically symmetric speed of sound}, SIAM J. Appl. Math., 54 (1994), 1203--1223.

\bibitem{Morioka--Shoji}
\textsc{Morioka, H.}, and \textsc{N. Shoji}: \textit{Interior transmission eigenvalue problems on compact manifolds with smooth boundary}, arXiv:1703.02704.

\bibitem{NIST}
\textit{NIST Digital Library of Mathematical Functions}, \url{http://dlmf.nist.gov/}, Release 1.0.9 of 2014-08-29, online companion to \cite{Olver--Lozier--Boisvert--Clark}.

\bibitem{Olver--Lozier--Boisvert--Clark}
\textsc{Olver, F.\ W.\ J.}, \textsc{D.\ W.\ Lozier}, \textsc{R.\ F.\ Boisvert}, and \textsc{C.\ W.\ Clark} (eds.):
\textit{NIST Handbook of Mathematical Functions,}
Cambridge University Press, 2010, print companion to \cite{NIST}.

\bibitem{Paivarinta--Salo--Vesalainen}
\textsc{P\"aiv\"arinta, L.}, \textsc{M. Salo} and \textsc{E. V\!. Vesalainen}:
\textit{Strictly convex corners scatter}, Rev. Mat. Iberoam., 33 (2017), 1369--1396.

\bibitem{Paivarinta--Sylvester}
\textsc{P\"aiv\"arinta, L.}, and \textsc{J. Sylvester}: \textit{Transmission eigenvalues}, SIAM J. Math. Anal., 40 (2008), 738--753.

\bibitem{Serrin}\textsc{Serrin, J.}: \textit{Removable singularities of solutions of elliptic equations}, Arch. Ration. Mech. Anal., 17 (1964), 67--78.

\bibitem{Shoji}
\textsc{Shoji, N.}: \textit{On $T$-coercive interior transmission eigenvalue problems on compact manifolds with smooth boundary}, Tsukuba J. Math., 41 (2017), 215--233.

\bibitem{Sylvester--Uhlmann}
\textsc{Sylvester, J.}, and \textsc{G. Uhlmann}: \textit{A global uniqueness theorem for an inverse boundary value problem}, Ann. of Math. (2), 125 (1987), 153--169.

\bibitem{Triebel1}
\textsc{Triebel, H.}: \textit{Theory of Function Spaces}, Birkh\"auser Verlag, 1983.

\bibitem{Triebel2}
\textsc{Triebel, H.}: \textit{Theory of Function Spaces II}, Birkh\"auser Verlag, 1992.

\bibitem{Uhlmann}
\textsc{Uhlmann, G.} (ed.): \textit{Inverse Problems and Applications: Inside Out II}, MSRI Publications, 60, Cambridge University Press, 2013.

\bibitem{Vesalainen1}
\textsc{Vesalainen, E.\ V.}: \textit{Transmission eigenvalues for a class of non-compactly supported potentials}, Inverse Problems, 29 (2013), 104006, 1--11.

\bibitem{Vesalainen2}\textsc{Vesalainen, E.\ V.}: \textit{Rellich type theorems for unbounded domains}, Inverse Probl.\ Imaging, 8 (2014), 865--883.

\end{thebibliography}
\end{document}